\documentclass[reqno,12pt]{amsart}
\usepackage{amsfonts}
\usepackage{bbm}
\usepackage{}
\numberwithin{equation}{section}
\usepackage{color}

\usepackage{indentfirst}
\usepackage{color}
\usepackage{amssymb}
\usepackage{mathrsfs}
\usepackage{xy}
\xyoption{all}
\def\Ext{\mbox{\rm Ext}\,} \def\Hom{\mbox{\rm Hom}}  
 \def\fin{\hfill$\square$}   \def\id{\mbox{\rm id}}
    
\def\End{\mbox{\rm End}\,}

\def\RR{\mbox{\rm RR}}

\theoremstyle{plain}
\newtheorem{theorem}{\bf Theorem}[section]
\newtheorem{lemma}[theorem]{\bf Lemma}
\newtheorem{corollary}[theorem]{\bf Corollary}
\newtheorem{proposition}[theorem]{\bf Proposition}

\theoremstyle{definition}
\newtheorem{definition}[theorem]{\bf Definition}
\newtheorem{remark}[theorem]{\bf Remark}
\newtheorem{example}[theorem]{\bf Example}

\newcommand{\bt}{\begin{theorem}}
\newcommand{\et}{\end{theorem}}
\newcommand{\bl}{\begin{lemma}}
\newcommand{\el}{\end{lemma}}
\newcommand{\bd}{\begin{definition}}
\newcommand{\ed}{\end{definition}}
\newcommand{\bc}{\begin{corollary}}
\newcommand{\ec}{\end{corollary}}
\newcommand{\bp}{\begin{proof}}
\newcommand{\ep}{\end{proof}}
\newcommand{\bx}{\begin{example}}
\newcommand{\ex}{\end{example}}
\newcommand{\br}{\begin{remark}}
\newcommand{\er}{\end{remark}}
\newcommand{\be}{\begin{equation}}
\newcommand{\ee}{\end{equation}}
\newcommand{\ba}{\begin{align}}
\newcommand{\ea}{\end{align}}
\newcommand{\bn}{\begin{enumerate}}
\newcommand{\en}{\end{enumerate}}
\newcommand{\bcs}{\begin{cases}}
\newcommand{\ecs}{\end{cases}}

\makeatletter
\renewcommand{\section}{\@startsection{section}{1}{0mm}
  {-\baselineskip}{0.5\baselineskip}{\bf\leftline}}
\makeatother

\begin{document}

\title[Idempotent completion of extriangulated categories]{Idempotent completion of extriangulated categories}
\author[L. Wang, J. Wei, H. Zhang]{Li Wang, Jiaqun Wei,  Haicheng Zhang}
\address{Institute of Mathematics, School of Mathematical Sciences, Nanjing Normal University,
 Nanjing 210023, P. R. China.\endgraf}
\email{wl04221995@163.com (Wang); weijiaqun@njnu.edu.cn (Wei); zhanghc@njnu.edu.cn (Zhang).}

\author[T. Zhao]{Tiwei Zhao}
\address{School of Mathematical Sciences, Qufu Normal University, Qufu 273165, P. R. China}
\email{tiweizhao@qfnu.edu.cn (Zhao)}


\begin{abstract}
Extriangulated categories were introduced by Nakaoka and Palu as a simultaneous generalization of exact categories and triangulated categories.  In this paper, we show that the idempotent completion of an extriangulated category admits a natural extriangulated structure.  As applications, we prove that cotorsion pairs in an extriangulated category induce cotorsion pairs in its idempotent completion under certain condition, and the idempotent completion of a recollement of extriangulated categories is still a  recollement.
\end{abstract}

\maketitle

\section{Introduction}
The concept of a triangulated category was given by Verdier and Grothendieck in
\cite{Ve} and has been investigated in many papers such as \cite{Bel,Ha,Iy2,Th}. Exact category and triangulated category are two fundamental structures in algebra and geometry. Recently, Nakaoka and Palu \cite{Na} introduced an extriangulated category which is extracting properties on triangulated categories and exact categories. The class of extriangulated categories contains triangulated categories and exact categories as examples, it is also closed under taking some ideal quotients. There have been many further researches on extriangulated categories, see for example, \cite{LiuY,Zhou1,Zhou2,Zhao}. Balmer and Schlichting \cite{Ba} showed that the idempotent completion of a triangulated category admits a triangulated structure. B\"{u}hler \cite{Bu} showed that the idempotent completion of an exact category is still an exact category. So one may ask whether the idempotent completion of an extriangulated category carries an extriangulated  structure. In this paper, we intend to give a positive answer, and thus it provides a unified framework for dealing with the idempotent completions of triangulated categories and exact categories.

Compared with exact categories or triangulated categories, the proof for an extriangulated category seems to be more complicated, because we have to construct a biadditive functor and its additive realization for the idempotent completion of an extriangulated category.

The notion of cotorsion pairs has been generalized to extriangulated categories in \cite{Na}.
Recollements of triangulated categories were introduced by Beilinson, Bernstein and Deligne \cite{BBD} in
connection with derived categories of sheaves on topological spaces with the idea that one triangulated category may be ``glued together" from two others. In \cite{CT}, Chen and Tang showed that the idempotent completion of a right (resp. left) recollement of triangulated categories is still a right (resp. left) recollement. As applications, we first introduce the right (left) recollements of extriangulated categories, and study the cotorsion pairs and recollements in extriangulated categories under the idempotent completion.

The paper is organized as follows: We summarize some basic definitions and propositions about idempotent completions and extriangulated categories in Section 2.
In Section 3, let $(\mathscr{C},\mathbb{E},\mathfrak{s})$ be a pre-extriangulated category, in the sense that it satisfies all axioms of an extriangulated category except $\rm(ET4)$ and $\rm(ET4)^{op}$. We show that the idempotent completion of $(\mathscr{C},\mathbb{E},\mathfrak{s})$ has a natural structure of a pre-extriangulated category. Furthermore, we extend this result to an extriangulated category. Section 4 is devoted to the cotorsion pairs and right (left) recollements in the idempotent completion of an extriangulated category. Explicitly, we prove that cotorsion pairs in an extriangulated category induce cotorsion pairs in its idempotent completion under certain condition, and the idempotent completion of a right (resp. left) recollement of extriangulated categories is still a right (resp. left) recollement.

\section{Preliminaries}

Throughout this paper, we assume, unless otherwise stated, that all considered categories are {additive}, and subcategories are {full}.

\subsection{Idempotent completions}
An additive category $\mathscr{C}$ is said to be {\em idempotent complete} if any idempotent morphism $e: A\rightarrow A$ {splits}. That is to say, there are two morphisms $p:A\rightarrow B$ and $q:B\rightarrow A$ such that $e=qp$ and $pq={\rm id}_{B}$. Remark that $\mathscr{C}$ is idempotent complete if and only if every idempotent morphism has a kernel, or every idempotent morphism has a cokernel.

\begin{definition}
 Let $\mathscr{C}$ be an additive category. The {\em idempotent completion} of $\mathscr{C}$ is the category $\widetilde{\mathscr{C}}$ defined as follows: Objects of $\widetilde{\mathscr{C}}$ are pairs $\widetilde{A}=(A,e_{a})$, where $A$ is an object of $\mathscr{C}$ and $e_{a}: A\rightarrow A$ is an idempotent morphism. A morphism in $\widetilde{\mathscr{C}}$ from $(A,e_{a})$ to $(B,e_{b})$ is a morphism $\alpha: A\rightarrow B$ in $\mathscr{C}$ such that $\alpha e_{a}=e_{b}\alpha=\alpha$.
\end{definition}

The assignment $A\mapsto (A,\id_{A})$ defines a functor $\iota: \mathscr{C}\rightarrow\widetilde{ \mathscr{C}}$. Obviously, the functor $\iota$ is fully faithful. Namely, we can view the category $\mathscr{C}$ as a full subcategory of $\widetilde{ \mathscr{C}}$.

\begin{remark}
(a) Note that triangulated categories with bounded t-structures are idempotent complete (cf. \cite{Ju}).

(b) The bounded derived category of an idempotent complete exact category is idempotent complete (cf. \cite{Ba}).

(c) Let $\mathscr{C}$ be the category of finite generated free modules over a ring $R$, its idempotent completion $\widetilde{\mathscr{C}}$ is equivalent to the category of finite generated projective modules over $R$ (cf. \cite{Li}).


\end{remark}

The following result is well-known (see, for example, \cite[Proposition 1.3]{Ba}).

\begin{lemma}\label{third}
The category $\widetilde{ \mathscr{C}}$ is an additive category, the functor $\iota: \mathscr{C}\rightarrow\widetilde{ \mathscr{C}}$ is additive and $\widetilde{ \mathscr{C}}$ is idempotent complete. Moreover, the functor$\iota$ induces an equivalence
$$\xymatrix{\Hom_{add}(\widetilde{ \mathscr{C}},L)\ar[r]^-{\sim}& \Hom_{add}(\mathscr{C},L)}$$
for each idempotent complete additive category $L$, where $\Hom_{add}$ denotes the category of additive functors.
 \end{lemma}

\subsection{Extriangulated categories}
Let us recall some definitions and notations concerning extriangulated categories. For more definitions and details, we refer the reader to \cite{Na}.

Let $\mathscr{C}$ be an additive category and let $\mathbb{E}$: $\mathscr{C}^{op}\times\mathscr{C}\rightarrow Ab$ be a biadditive functor, where $Ab$ denotes the category of abelian groups. For any pair of objects $A$, $C\in\mathscr{C}$, an element $\delta\in \mathbb{E}(C,A)$ is called an {\em $\mathbb{E}$-extension}. Thus formally, an $\mathbb{E}$-extension is a triplet $(A,\delta,C)$. The zero element $0\in\mathbb{E}(C,A)$ is called the {\em split $\mathbb{E}$-extension}.

Let $\delta\in \mathbb{E}(C,A)$ be any $\mathbb{E}$-extension. By the functoriality of $\mathbb{E}$, for any $a\in \mathscr{C}(A,A')$ and $c\in {\mathscr{C}}(C',C)$, we have
$$\mathbb{E}(C,a)(\delta)\in\mathbb{E}(C,A')~~~~\text{and}~~~~\mathbb{E}(c,A)(\delta)\in\mathbb{E}(C',A).$$
We simply denote them by $a_{\ast}\delta$ and $c^{\ast}\delta$,~respectively. In this terminology, we have $\mathbb{E}(c,a)(\delta)=c^{\ast}a_{\ast}\delta=a_{\ast}c^{\ast}\delta$ in $\mathbb{E}(C',A')$. A morphism $(a,c)$: $(A,\delta,C)\rightarrow(A',\delta',C')$ of $\mathbb{E}$-extensions is a pair of morphisms $a\in \mathscr{C}(A,A')$ and $c\in {\mathscr{C}}(C,C')$ satisfying the equality
$$a_{\ast}\delta=c^{\ast}\delta',$$
which is simply denoted by $(a,c)$: $\delta\rightarrow\delta'$. We obtain the category $\mathbb{E}$-Ext($\mathscr{C}$) of $\mathbb{E}$-extensions, with composition and identities naturally induced from those in $\mathscr{C}$.

By Yoneda's lemma, any $\mathbb{E}$-extension $\delta\in \mathbb{E}(C,A)$ induces natural transformations
$$\delta_{\sharp}: \mathscr{C}(-,C)\rightarrow\mathbb{E}(-,A)~~\text{and}~~\delta^{\sharp}: \mathscr{C}(A,-)\rightarrow\mathbb{E}(C,-).$$ For any $X\in\mathscr{C}$, these
$(\delta_{\sharp})_X$ and $(\delta^{\sharp})_X$ are defined by
$(\delta_{\sharp})_X:  \mathscr{C}(X,C)\rightarrow\mathbb{E}(X,A), f\mapsto f^\ast\delta$ and $(\delta^{\sharp})_X:   \mathscr{C}(A,X)\rightarrow\mathbb{E}(C,X), g\mapsto g_\ast\delta$.

Two sequences of morphisms $A\stackrel{x}{\longrightarrow}B\stackrel{y}{\longrightarrow}C$ and $A\stackrel{x'}{\longrightarrow}B'\stackrel{y'}{\longrightarrow}C$ in $\mathscr{C}$ are said to be {\em equivalent} if there exists an isomorphism $b\in \mathscr{C}(B,B')$ such that the following diagram
$$\xymatrix{
  A \ar@{=}[d] \ar[r]^-{x} & B\ar[d]_{b}^-{\simeq} \ar[r]^-{y} & C\ar@{=}[d] \\
  A \ar[r]^-{x'} &B' \ar[r]^-{y'} &C  }$$ is commutative.
We denote the equivalence class of $A\stackrel{x}{\longrightarrow}B\stackrel{y}{\longrightarrow}C$ by $[A\stackrel{x}{\longrightarrow}B\stackrel{y}{\longrightarrow}C]$. In addition, for any $A,C\in\mathscr{C}$, we denote as
$$0=[A\stackrel{\tiny\begin{pmatrix} 1\\0\end{pmatrix}}{\longrightarrow}A\oplus C\stackrel{\tiny\begin{pmatrix} 0&1\end{pmatrix}}{\longrightarrow}C].$$
For any two classes $[A\stackrel{x}{\longrightarrow}B\stackrel{y}{\longrightarrow}C]$ and $[A'\stackrel{x'}{\longrightarrow}B'\stackrel{y'}{\longrightarrow}C']$, we denote as
$$[A\stackrel{x}{\longrightarrow}B\stackrel{y}{\longrightarrow}C]\oplus[A'\stackrel{x'}{\longrightarrow}B'\stackrel{y'}{\longrightarrow}C']=
[A\oplus A'\stackrel{x\oplus x'}{\longrightarrow}B\oplus B'\stackrel{y\oplus y'}{\longrightarrow}C\oplus C'].$$
Let $\delta\in\mathbb{E}(C,A)$, $\delta'\in\mathbb{E}(C',A')$ be any pair of $\mathbb{E}$-extensions. Let $C\stackrel{l_{C}}{\longrightarrow}C\oplus C'\stackrel{l_{C'}}{\longleftarrow}C'$ and $A\stackrel{p_{A}}{\longleftarrow}A\oplus A'\stackrel{p_{A'}}{\longrightarrow}A'$ be the coproduct and product in $\mathscr{C}$, respectively. Then the biadditivity of $\mathbb{E}$ implies
\begin{equation}\label{zhkz}\mathbb{E}(C\oplus C',A\oplus A')\cong\mathbb{E}(C,A)\oplus\mathbb{E}(C,A')\oplus\mathbb{E}(C',A)\oplus\mathbb{E}(C',A').\end{equation}
Thus we set $\delta\oplus\delta'=(\delta,0,0,\delta')\in\mathbb{E}(C\oplus C',A\oplus A')$ through this isomorphism. In particular, if $C=C'$ and $A=A'$, then the sum
\begin{equation}\label{cha}\delta+\delta'=\Delta^{\ast}\nabla_{\ast}(\delta\oplus\delta')\in\mathbb{E}(C,A)\end{equation}
where $\Delta=\footnotesize\begin{pmatrix} 1 \\1 \end{pmatrix}: C\rightarrow C\oplus C$, and $\nabla=(1,1):A\oplus A\rightarrow A$.

In what follows, we write an element in (\ref{zhkz}) in the form of matrices. For example, we will write $\delta=\footnotesize\begin{pmatrix} \delta_{1}&\delta_{3}\\ \delta_{2}&\delta_{4} \end{pmatrix}$ for $\delta=(\delta_1,\delta_2,\delta_3,\delta_4)\in\mathbb{E}(C\oplus C',A\oplus A')$.

\begin{definition}
Let $\mathfrak{s}$ be a correspondence which associates an equivalence class $\mathfrak{s}(\delta)=[A\stackrel{x}{\longrightarrow}B\stackrel{y}{\longrightarrow}C]$ to any $\mathbb{E}$-extension $\delta\in\mathbb{E}(C,A)$ . This $\mathfrak{s}$ is called a {\em realization} of $\mathbb{E}$ if for any morphism $(a,c):\delta\rightarrow\delta'$ with $\mathfrak{s}(\delta)=[\Delta_{1}]$ and $\mathfrak{s}(\delta')=[\Delta_{2}]$, there is a commutative diagram as follows:
$$\xymatrix{
\Delta_{1}\ar[d] & A \ar[d]_-{a} \ar[r]^-{x} & B  \ar[r]^{y}\ar[d]_-{b} & C \ar[d]_-{c}    \\
 \Delta_{2}&A\ar[r]^-{x'} & B \ar[r]^-{y'} & C .   }
$$  A realization $\mathfrak{s}$ of $\mathbb{E}$ is said to be {\em additive} if it satisfies the following conditions:

(a) For any $A,~C\in\mathscr{C}$, the split $\mathbb{E}$-extension $0\in\mathbb{E}(C,A)$ satisfies $\mathfrak{s}(0)=0$.

(b) $\mathfrak{s}(\delta\oplus\delta')=\mathfrak{s}(\delta)\oplus\mathfrak{s}(\delta')$ for any pair of $\mathbb{E}$-extensions $\delta$ and $\delta'$.
\end{definition}

Let $\mathscr{C}$ be an additive category and $\mathfrak{s}$ be an additive realization of $\mathbb{E}$.

$\bullet$ We call a sequence $A\stackrel{x}{\longrightarrow}B\stackrel{y}{\longrightarrow}C$ a {\em conflation} if its equivalence class is equal to  $\mathfrak{s}(\delta)$ for some $\delta\in\mathbb{E}(C,A)$ and we say that this conflation realizes $\delta$. In this case, we call the pair $(A\stackrel{x}{\longrightarrow}B\stackrel{y}{\longrightarrow}C,\delta)$ an $\mathbb{E}$-{\em triangle}, and write it in the following way.
$$A\stackrel{x}{\longrightarrow}B\stackrel{y}{\longrightarrow}C\stackrel{\delta}\dashrightarrow$$

$\bullet$ Let $\mathcal{D}\subseteq\mathscr{C}$ be a full additive subcategory, which is closed under isomorphisms. The subcategory $\mathcal{D}$ is said to be {\em extension-closed}, if for any conflation $A\stackrel{x}{\longrightarrow}B\stackrel{y}{\longrightarrow}C$ which satisfies $A,C\in\mathcal{D}$, then $B\in\mathcal{D}$.

$\bullet$ For any morphism $(a,c):\delta\rightarrow\delta'$ with $\delta\in\mathbb{E}(C,A)$ and $\delta'\in\mathbb{E}(C',A')$, there is a commutative diagram with rows being $\mathbb{E}$-triangles:
$$\xymatrix{
  A \ar[d]_{a} \ar[r]^{x} & B  \ar[r]^{y}\ar[d]_{b} & C \ar[d]_{c} \ar@{-->}[r]^{\delta} &  \\
 A'\ar[r]^{x'} & B' \ar[r]^{y'} & C' \ar@{-->}[r]^{\delta'} &    }
$$
We call $(a,b,c)$ a {\em morphism of $\mathbb{E}$-triangles}, or a {\em realization} of $(a,c):\delta\rightarrow\delta'$. 

\begin{definition}
(\cite[Definition 2.12]{Na})
We call the triplet $(\mathscr{C}, \mathbb{E},\mathfrak{s})$ an {\em extriangulated category} if it satisfies the following conditions:\\
$\rm(ET1)$ $\mathbb{E}$: $\mathscr{C}^{op}\times\mathscr{C}\rightarrow Ab$ is a biadditive functor.\\
$\rm(ET2)$ $\mathfrak{s}$ is an additive realization of $\mathbb{E}$.\\
$\rm(ET3)$ Let $\delta\in\mathbb{E}(C,A)$ and $\delta'\in\mathbb{E}(C',A')$ be any pair of $\mathbb{E}$-extensions, realized as
$\mathfrak{s}(\delta)=[A\stackrel{x}{\longrightarrow}B\stackrel{y}{\longrightarrow}C]$, $\mathfrak{s}(\delta')=[A'\stackrel{x'}{\longrightarrow}B'\stackrel{y'}{\longrightarrow}C']$. For any commutative square in $\mathscr{C}$
$$\xymatrix{
  A \ar[d]_{a} \ar[r]^{x} & B \ar[d]_{b} \ar[r]^{y} & C \\
  A'\ar[r]^{x'} &B'\ar[r]^{y'} & C'}$$
there exists a morphism $(a,c)$: $\delta\rightarrow\delta'$ which is realized by $(a,b,c)$.\\
$\rm(ET3)^{op}$~Dual of $\rm(ET3)$.\\
$\rm(ET4)$~Let $\delta\in\mathbb{E}(D,A)$ and $\delta'\in\mathbb{E}(F,B)$ be $\mathbb{E}$-extensions realized by
$A\stackrel{f}{\longrightarrow}B\stackrel{f'}{\longrightarrow}D$ and $B\stackrel{g}{\longrightarrow}C\stackrel{g'}{\longrightarrow}F$, respectively.
Then there exist an object $E\in\mathscr{C}$, a commutative diagram
$$\xymatrix{
  A \ar@{=}[d]\ar[r]^-{f} &B\ar[d]_-{g} \ar[r]^-{f'} & D\ar[d]^-{d} \\
  A \ar[r]^-{h} & C\ar[d]_-{g'} \ar[r]^-{h'} & E\ar[d]^-{e} \\
   & F\ar@{=}[r] & F   }$$
in $\mathscr{C}$, and an $\mathbb{E}$-extension $\delta''\in \mathbb{E}(E,A)$ realized by $A\stackrel{h}{\longrightarrow}C\stackrel{h'}{\longrightarrow}E$, which satisfy the following compatibilities:\\
$(\textrm{i})$ $D\stackrel{d}{\longrightarrow}E\stackrel{e}{\longrightarrow}F$ realizes $\mathbb{E}(F,f')(\delta')$,\\
$(\textrm{ii})$ $\mathbb{E}(d,A)(\delta'')=\delta$,\\
$(\textrm{iii})$ $\mathbb{E}(E,f)(\delta'')=\mathbb{E}(e,B)(\delta')$.\\
$\rm(ET4)^{op}$ Dual of $\rm(ET4)$.
\end{definition}

\begin{example}\label{exa}
(a)  Exact categories, triangulated categories and extension-closed subcategories of an extriangulated category are
extriangulated categories (cf. \cite{Na}).

(b) Let $\mathscr{C}$ be an extriangulated category. An object $P$ in $\mathscr{C}$ is called {\em projective} if for any $\mathbb{E}$-triangle $A\stackrel{x}{\longrightarrow}B\stackrel{y}{\longrightarrow}C\stackrel{\delta}\dashrightarrow$ and any morphism $c$ in $\mathscr{C}(P,C)$, there exists $b$ in $\mathscr{C}(P,B)$ such that $yb=c$. We denote the full subcategory of projective objects in $\mathscr{C}$ by $\mathcal{P}$. Dually, the {\em injective} objects are defined, and the full subcategory of injective objects in $\mathscr{C}$ is denoted by $\mathcal{I}$. Then $\mathscr{C}/(\mathcal{P}\cap \mathcal{I})$ is an extriangulated category which is neither exact nor triangulated in general (cf. \cite[Proposition 3.30]{Na}).
\end{example}
\begin{definition} We call the triplet $(\mathscr{C},\mathbb{E},\mathfrak{s})$ a {\em pre-extriangulated category} if it satisfies all except $\rm(ET4)$ and $\rm(ET4)^{op}$.
\end{definition}

There is a natural way to produce pre-extriangulated categories. In detail, let $(\mathscr{C},\mathbb{E},\mathfrak{s})$ be an extriangulated category and let $\widehat{\mathbb{E}}$ be an additive subfunctor of $\mathbb{E}$. Then for any $a\in\mbox{Hom}_{\mathscr{C}}(A,A')$, $c\in\mbox{Hom}_{\mathscr{C}}(C',C)$, and $\delta\in\widehat{\mathbb{E}}(C,A)$, we have
 $a_\ast\delta\in\widehat{\mathbb{E}}(C,A')$ and $c^\ast \delta\in\widehat{\mathbb{E}}(C',A)$.
 Define $\widehat{\mathfrak{s}}$ to be the restriction of $\mathfrak{s}$ to $\widehat{\mathbb{E}}$, that is, it is defined by
$\widehat{\mathfrak{s}}(\delta)=\mathfrak{s}(\delta)$ for any $\widehat{\mathbb{E}}$-extension $\delta$. By \cite[Claim 3.8]{HLN}, the triplet $(\mathscr{C},\widehat{\mathbb{E}},\widehat{\mathfrak{s}})$ satisfies (ET1), (ET2), (ET3) and (ET3)$^{\tiny\mbox{op}}$, that is, $(\mathscr{C},\widehat{\mathbb{E}},\widehat{\mathfrak{s}})$  is a pre-extriangulated category.

In fact, the following example shows that the category of finitely generated modules over an Artin algebra always admits a pre-extriangulated structure which is not extriangulated.

\begin{example}\label{e2.7}
Let $\Lambda$ be an Artin algebra, and $\mathscr{C}={\rm mod} \Lambda$ the category of finitely generated $\Lambda$-modules.
Let $\mathbb{E}(C,A)=\mbox{Ext}^1_\Lambda(C,A)$ for any $A,C\in{\rm mod} \Lambda$, and define   $\widehat{\mathbb{E}}(C,A)$ to be the subgroup of $\Ext^1_\Lambda(C,A)$, which is generated by the corresponding almost split sequences in ${\rm mod} \Lambda$. Then $(\mathscr{C},\widehat{\mathbb{E}},\widehat{\mathfrak{s}})$ is a pre-extriangulated
category but not an extriangulated
category, see \cite{DRSSK} for more details.
\end{example}

\begin{proposition} \cite[Proposition 3.3]{Na}
Let $(\mathscr{C},\mathbb{E},\mathfrak{s})$ be a pre-extriangulated category. For any $\mathbb{E}$-triangle $A\stackrel{x}{\longrightarrow}B\stackrel{y}{\longrightarrow}C\stackrel{\delta}\dashrightarrow$, the following sequences of natural transformations are exact.
$$\mathscr{C}(C,-)\rightarrow \mathscr{C}(B,-)\rightarrow \mathscr{C}(A,-)\stackrel{\delta^{\sharp}}\rightarrow\mathbb{E}(C,-)\rightarrow\mathbb{E}(B,-),$$
$$\mathscr{C}(-,A)\rightarrow \mathscr{C}(-,B)\rightarrow \mathscr{C}(-,C)\stackrel{\delta_{\sharp}}\rightarrow\mathbb{E}(-,A)\rightarrow\mathbb{E}(-,B).$$
\end{proposition}

Let us give two preliminary lemmas.

\begin{lemma}\label{first} Let $(\mathscr{C},\mathbb{E},\mathfrak{s})$ be a pre-extriangulated category. Consider the following commutative diagram in $\mathscr{C}$
$$\xymatrix{
  A \ar[d]_-{p} \ar[r]^-{u} & B  \ar[r]^-{v} \ar[d]^-{b}& C \ar[d]_-{q} \ar@{-->}[r]^-{\delta} &  \\
 A\ar[r]^-{u} & B \ar[r]^-{v} & C \ar@{-->}[r]^-{\delta} &    }
$$ in which the rows are $\mathbb{E}$-triangles. Suppose that any two of the triplet $(p,b,q)$ are idempotents, then the third morphism can be replaced by an idempotent morphism such that the diagram above is still commutative.

\end{lemma}
\begin{proof}
First of all, suppose that the first and third morphisms are idempotents. Then $b^{2}u=bup=up^{2}=up$ and similarly $vb^{2}=qv$. Set $h=b^{2}-b$, then there exists a morphism $s:C\rightarrow B$ such that $sv=h$ since $hu=0$. It follows that $h^{2}=svh=sv(b^{2}-b)=0$. Now, let $r=b+h-2bh$~and note that $bh=hb$,~then $r^{2}=b^{2}+2bh-4b^2h=b+h+2bh-4(b+h)h=b+h-2bh=r$. However, $ru=(b+h-2bh)u=bu=up$ and $vr=v(b+h-2bh)=qv+vh-2qvh=qv$. Hence, we obtain an idempotent morphism $r$ such that the above diagram is commutative.

Now, suppose the first two morphisms are idempotents. It should be noted that $(q^{2})^{\ast}\delta=q^{\ast}p_{\ast}\delta=p_{\ast}q^{\ast}\delta=(p^{2})_{\ast}\delta=p_{\ast}\delta$. Similar to the previous proof, set $\eta=q^{2}-q$ and $\sigma=q+\eta-2\eta h$, we can obtain an idempotent morphism $\sigma$ such that $(p,b,\sigma)$ realizes $(p,\sigma):\delta\rightarrow\delta$. Similarly, we can prove the case that the second and third morphisms are idempotents.
\end{proof}

\begin{lemma}\label{second}
Let $(\mathscr{C},\mathbb{E},\mathfrak{s})$ be a pre-extriangulated category. A morphism sequence
$$\bigtriangleup:~~~~~~~A\oplus A'\stackrel{\tiny\begin{pmatrix} x&0\\0&x' \end{pmatrix}}{\longrightarrow}B\oplus B'\stackrel{\tiny\begin{pmatrix} y&0 \\0&y' \end{pmatrix}}{\longrightarrow}C\oplus C'\stackrel{\tiny\begin{pmatrix} \delta_{1}&0 \\0&\delta_{2} \end{pmatrix}}\dashrightarrow$$
is an $\mathbb{E}$-triangle for $\delta_{1}\oplus\delta_{2}$, where $\delta_{1}\in\mathbb{E}(C,A)$ and $\delta_{2}\in\mathbb{E}(C',A')$, if and only if
$$~~~~\triangle_{1}:~A\stackrel{x}{\longrightarrow}B\stackrel{y}{\longrightarrow}C\stackrel{\delta_{1}}\dashrightarrow~~and~~~
\triangle_{2}:~A'\stackrel{x'}{\longrightarrow}B'\stackrel{y'}{\longrightarrow}C'\stackrel{\delta_{2}}\dashrightarrow$$
are $\mathbb{E}$-triangles for $\delta_{1}$ and $\delta_{2}$, respectively.
\end{lemma}
\begin{proof}
Suppose $\bigtriangleup_{1}$ and $\bigtriangleup_{2}$ are $\mathbb{E}$-triangles. Then $\bigtriangleup$ is an $\mathbb{E}$-triangle since $\mathfrak{s}(\delta_{1})\oplus\mathfrak{s}(\delta_{2})=\mathfrak{s}(\delta_{1}\oplus\delta_{2})$. Conversely, suppose that the direct sum of $\triangle_{1}$ and $\triangle_{2}$ is an $\mathbb{E}$-triangle for $\delta_{1}\oplus\delta_{2}\in\mathbb{E}(C\oplus C',A\oplus A')$. Then we have respectively two $\mathbb{E}$-triangles for $\delta_{1}$ and $\delta_{2}$:

\begin{equation}\label{E1}A\stackrel{u}{\longrightarrow}D\stackrel{v}{\longrightarrow}C\stackrel{\delta_{1}}\dashrightarrow\end{equation}and
\begin{equation}\label{E2}A'\stackrel{u'}{\longrightarrow}E\stackrel{v'}{\longrightarrow}C'\stackrel{\delta_{2}}\dashrightarrow.
\end{equation}
It follows that we have a  commutative diagram with rows being $\mathbb{E}$-triangles:

$$
\xymatrix{
  A\oplus A'\ar[d]_-{(\id_{A}~0)} \ar[r]^-{\tiny\begin{pmatrix} x&0\\0&x' \end{pmatrix}} & B\oplus B' \ar[d]_-{(b_{1}~{b_{1}}')} \ar[r]^-{\tiny\begin{pmatrix} y&0 \\0&y' \end{pmatrix}} & C\oplus C' \ar[d]_-{(\id_{C}~0)} \ar@{-->}[r]^-{\tiny\begin{pmatrix} \delta_{1}&0 \\0&\delta_{2} \end{pmatrix}=\delta} &  \\
  A\ar[r]^-{u} & D \ar[r]^-{v} & C \ar@{-->}[r]^-{\delta_{1}} &  }$$
where the middle morphism exists since $(\id_{A},0)_{\ast}\delta=(\id_{C},0)^{\ast}\delta_{1}$. Then we have a morphism $b_{1}:B\rightarrow D$ such that $b_{1}x=u$ and $y=vb_{1}$. Similarly, we also have a morphism $b_{2}: B'\rightarrow E$ such that $b_{2}x'=u'$ and $y'=v'b_{2}$. Hence, we have the following commutative diagram of $\mathbb{E}$-triangles in $\mathscr{C}$
$$
\xymatrix{
  A\oplus A'\ar@{=}[d] \ar[r]^-{\tiny\begin{pmatrix} x&0\\0&x' \end{pmatrix}} & B\oplus B' \ar[d]^-{\tiny\begin{pmatrix} b_{1}&0\\0&b_{2} \end{pmatrix}} \ar[r]^-{\tiny\begin{pmatrix} y&0 \\0&y' \end{pmatrix}} & C\oplus C' \ar@{=}[d] \ar@{-->}[r]^-{\delta} &  \\
  A\oplus A'\ar[r]_-{\tiny\begin{pmatrix} u&0\\0&u' \end{pmatrix}} & B\oplus B'\ar[r]_-{\tiny\begin{pmatrix} v&0 \\0&v' \end{pmatrix}} &C\oplus C' \ar@{-->}[r]^-{\delta} &  .}$$
This implies that $b_{1}$ and $b_{2}$ are isomorphisms by \cite[Corollary 3.6]{Na}. Then, $\triangle_{1}$ and $\triangle_{2}$ are isomorphic to the $\mathbb{E}$-triangles $(\ref{E1})$ and $(\ref{E2})$, respectively. This finishes the proof.
\end{proof}

\section{Main results}

\begin{theorem}\label{main} The idempotent completion of an extriangulated category has a natural structure of an extriangulated category.
\end{theorem}

Before proving Theorem \ref{main}, we need some preparations.

$\mathbf{(Step~1)}$ First of all, we are going to construct a biadditive functor $$\mathbb{F}: \widetilde{\mathscr{C}}^{op}\times\widetilde{\mathscr{C}}\rightarrow Ab.$$ For any objects $\widetilde{C}=(C,e_{c})$ and $\widetilde{A}=(A,e_{a})$ in $\widetilde{\mathscr{C}}$ we define
$$\mathbb{F}(\widetilde{C},\widetilde{A})=\{\widetilde{\omega}=(\omega,e_{c},e_{a})\mid\omega\in\mathbb{E}(C,A)~\text{satisfies}~ e_{c}^{\ast}\omega=\omega=({e_{a}})_{\ast}\omega\}.$$
For two elements $\widetilde{\omega_{1}},\widetilde{\omega_{2}}\in\mathbb{F}(\widetilde{C},\widetilde{A})$, define \begin{equation}\label{jiafa}\widetilde{\omega_{1}}+\widetilde{\omega_{2}}=\widetilde{\omega_{1}+\omega_{2}}=(\omega_{1}+\omega_{2},e_{c},e_{a}).\end{equation}
In fact, $e_{c}^{\ast}(\omega_{1}+\omega_{2})=e_{c}^{\ast}\omega_{1}+e_{c}^{\ast}\omega_{2}=\omega_{1}+\omega_{2}$ and $(e_{a})_{\ast}(\omega_{1}+\omega_{2})=\omega_{1}+\omega_{2}$ since $\mathbb{E}$ is a biadditive functor. So the addition $(\ref{jiafa})$ is well-defined. Thus, we have an abelian group $\mathbb{F}(\widetilde{C},\widetilde{A})$.

For $\widetilde{A'}=(A',e_{a'})$ we construct a group homomorphism between $\mathbb{F}(\widetilde{C},\widetilde{A})$ and $\mathbb{F}(\widetilde{C},\widetilde{A'})$.  For any $\alpha\in\widetilde{\mathscr{C}}(\widetilde{A},\widetilde{A'})$ and $\widetilde{\omega}\in\mathbb{F}(\widetilde{C},\widetilde{A})$, we define a homomorphism $\alpha_{\bullet}$ by
$$\xymatrix{
  \mathbb{F}(\widetilde{C},\widetilde{A})\ar[d]_{\alpha_{\bullet}}  & \widetilde{\omega}=(\omega,e_{c},e_{a}) \ar@{|->}[d] \\
  \mathbb{F}(\widetilde{C},\widetilde{A'})  & \alpha_{\bullet}\widetilde{\omega}=(\alpha_{\ast}\omega,e_{c},e_{a'}).   }$$
Indeed, since $\alpha_{\ast}\omega=\alpha_{\ast}{e_{c}}^{\ast}\omega={e_{c}}^{\ast}\alpha_{\ast}\omega$ and $\alpha_{\ast}\omega=\alpha_{\ast}{(e_{a})}_{\ast}\omega=(\alpha e_{a})_{\ast}\omega=(e_{a}'\alpha)_{\ast}\omega={(e_{a}')}_{\ast}\alpha_{\ast}\omega$, we obtain that $\alpha_{\bullet}\widetilde{\omega}$ really belongs to $\mathbb{F}(\widetilde{C},\widetilde{A'})$. Similarly, we can define $\beta^{\bullet}$ such that $\beta^{\bullet}\widetilde{\omega}\in\mathbb{F}(\widetilde{C'},\widetilde{A})$ for $\beta\in\widetilde{\mathscr{C}}(\widetilde{C'},\widetilde{C}).$

On the one hand, if there is another homomorphism $\alpha'\in\widetilde{\mathscr{C}}(\widetilde{A},\widetilde{A'})$. It is easy to see
\begin{equation}
(\alpha+\alpha')_{\bullet}(\widetilde{\omega})=((\alpha+\alpha')_{\ast}\omega,e_{c},e_{a}')\\
=(\alpha_{\ast}\omega+\alpha'_{\ast}\omega,e_{c},e_{a}')=\alpha_{\bullet}\widetilde{\omega}+\alpha'_{\bullet}\widetilde{\omega}
\end{equation}
and so is $(\beta+\beta')^{\bullet}(\widetilde{\omega})=\beta^{\bullet}\widetilde{\omega}+\beta'^{{\bullet}}\widetilde{\omega}$ for $\beta'\in\widetilde{\mathscr{C}}(\widetilde{C'},\widetilde{C}).$

On the other hand, for any morphism $m:\widetilde{A}\rightarrow \widetilde{A'}$ and $n:\widetilde{C'}\rightarrow \widetilde{C}$. We have the following commutative diagram which follows from the fact that $\mathbb{E}$ is biadditive.
$$\xymatrix{
  \mathbb{F}(\widetilde{C},\widetilde{A})\ar[d]_-{n^{\bullet}} \ar[r]^-{m_{\bullet}} & \mathbb{F}(\widetilde{C},\widetilde{A'})\ar[d]^-{n^{\bullet}} \\
  \mathbb{F}(\widetilde{C'},\widetilde{A}) \ar[r]^-{m_{\bullet}}& \mathbb{F}(\widetilde{C'},\widetilde{A'})  }
$$
From what has been discussed above, we have the following

\begin{lemma}\label{1}
Suppose $\mathscr{C}$ is equipped with a biadditive functor $\mathbb{E}: \mathscr{C}^{op}\times\mathscr{C}\rightarrow Ab$. Then we have a biadditive functor $\mathbb{F}$ : $\widetilde{\mathscr{C}}^{op}\times\widetilde{\mathscr{C}}\rightarrow Ab.$
\end{lemma}

For any pair of objects $\widetilde{A}$, $\widetilde{B}\in\widetilde{\mathscr{C}}$, an element $\widetilde{\omega}\in\mathbb{F}(\widetilde{C},\widetilde{A})$ is called an $\mathbb{F}$-{\em extension}. A morphism $(\alpha,\beta)$: $\widetilde{\omega}\rightarrow\widetilde{\omega'}$ is a pair of morphisms $\alpha\in\widetilde{\mathscr{C} }(\widetilde{A},\widetilde{A'})$ and $\beta\in\widetilde{\mathscr{C}}(\widetilde{C},\widetilde{C'})$ such that $\alpha_{\bullet}\widetilde{\omega}=\beta^{\bullet}\widetilde{\omega'}$. We obtain the category $\mathbb{F}$-Ext$(\widetilde{\mathscr{C}})$ of $\mathbb{F}$-extensions. For any $\mathbb{E}$-extension $\delta\in\mathbb{E}(C,A)$, the assignment
$$\delta\mapsto\widetilde{\delta}=(\delta,\id_{C},\id_{A})$$
defines a fully faithful functor $\tilde{i}$ from $\mathbb{E}$-Ext$(\mathscr{C})$ to $\mathbb{F}$-Ext$(\widetilde{\mathscr{C}})$. In addition, for any $\gamma\in\mathscr{C}(A,A')$, $\gamma_{\bullet}\widetilde{\delta}=(\gamma_{\ast}\delta,\id_{C},\id_{A})=\gamma_{\ast}\delta$. Similarly, for any $\theta\in\mathscr{C}(C',C)$, we have $\theta^{\bullet}\widetilde{\delta}=\theta^{\ast}\delta$. Namely, $\mathbb{F}\mid_{\mathscr{C}^{op}\times\mathscr{C}}=\mathbb{E}$.\\

$\mathbf{(Step~2)}$ Secondly, we need to construct an additive realization $\mathfrak{c}$ for $\mathbb{F}$. For any $\mathbb{E}$-extension $\delta\in\mathbb{E}(C,A)$, there is a correspondence which associates an equivalence class $\mathfrak{s}(\delta)=[A\stackrel{x}{\longrightarrow}B\stackrel{y}{\longrightarrow}C]$. If $\widetilde{A}=(A,e_{a})\in\widetilde{\mathscr{C}}$, take $\widetilde{A'}=(A,\id_{A}$$-e_{a})$. It is easy to check that $\widetilde{A}\oplus\widetilde{A'}\cong(A,\id_{A})$. Similarly, for $\widetilde{C}=(C,e_{c})$, take $\widetilde{C'}=(C,\id_C-e_{c})$, we have $\widetilde{C}\oplus\widetilde{C'}\cong(C,\id_{C})$. For any $\widetilde{\omega}\in\mathbb{F}(\widetilde{C},\widetilde{A})$, since $\footnotesize\begin{pmatrix} \widetilde{\omega}&0\\0&0\end{pmatrix}$ belongs to $\mathbb{F}(\widetilde{C}\oplus \widetilde{C'},\widetilde{A}\oplus \widetilde{A'})\cong\mathbb{E}(C,A)$, there exists an $\mathbb{E}$-triangle in $\mathscr{C}$

$$\Delta:~~~~~~~~~~\widetilde{A}\oplus \widetilde{A'}\stackrel{\alpha}{\longrightarrow}B\stackrel{\beta}{\longrightarrow}\widetilde{C}\oplus \widetilde{C'}\stackrel{\tiny\begin{pmatrix} \widetilde{\omega}&0\\0&0 \end{pmatrix}}\dashrightarrow.$$
By Lemma \ref{first}, there exists an idempotent morphism $b:B\rightarrow B$ such that the following is the morphism of $\mathbb{E}$-triangles
$$
\xymatrix{
  \widetilde{A}\oplus \widetilde{A'}\ar[d]_-{ \tiny\begin{pmatrix}e_{a}&0\\0&0 \end{pmatrix}} \ar[r]^-{\alpha} & B\ar[d]_-{b} \ar[r]^-{\beta~~} & \widetilde{C}\oplus \widetilde{C'}\ar@{-->}[r]^-{\tiny\begin{pmatrix} \widetilde{\omega}&0\\0&0 \end{pmatrix}} \ar[d]_-{ \tiny\begin{pmatrix} e_{c}&0\\0&0 \end{pmatrix}}&  \\
  \widetilde{A}\oplus\widetilde{ A'} \ar[r]^-{\alpha} & B \ar[r]^-{\beta~~} & \widetilde{C}\oplus \widetilde{C'}\ar@{-->}[r]^-{\tiny\begin{pmatrix} \widetilde{\omega}&0\\0&0 \end{pmatrix}} &  .}
$$
Consider the sequence $\Delta_{1}:\widetilde{A}\stackrel{b\alpha}{\longrightarrow}(B,b)\stackrel{\beta b}{\longrightarrow}\widetilde{C}$ in $\widetilde{\mathscr{C}}$, we have the canonical inclusions

\begin{equation}\label{remark}
\xymatrix{
 \widetilde{A} \ar@{^(->}[d] \ar[r]^-{b\alpha} & (B,b) \ar@{^(->}[d] \ar[r]^-{\beta b} & \widetilde{C} \ar@{^(->}[d] & \\
 \widetilde{A}\oplus \widetilde{A'} \ar[r]^-{\alpha} & (B, \id_{B}) \ar[r]^-{\beta} & \widetilde{C}\oplus \widetilde{C'}& }
\end{equation}
and $(\id_{\widetilde{A}},0)_{\bullet}\widetilde{\omega}=(\id_{\widetilde{C}},0)^{\bullet}\tiny\begin{pmatrix} \widetilde{\omega}&0\\0&0 \end{pmatrix}$. We set $\mathfrak{c}(\widetilde{0})=\widetilde{0}$ and $\mathfrak{c}(\widetilde{\omega})=[\Delta_1]$. Then $\mathfrak{c}(\widetilde{\omega}\oplus\widetilde{\omega'})=\mathfrak{c}(\widetilde{\omega})\oplus{\mathfrak{c}}(\widetilde{\omega}')$. Thus $\mathfrak{c}$ is a correspondence which associates to $\widetilde{\omega}\in\mathbb{F}(\widetilde{C},\widetilde{A})$ an equivalence class $[\widetilde{A}\stackrel{x}{\longrightarrow}\widetilde{B}\stackrel{y}{\longrightarrow}\widetilde{C}]$ in $\widetilde{\mathscr{C}}$. In this case, we call the sequence an {\em $\mathbb{F}$-triangle} and write it in the following way
$$\Delta_{2}:\widetilde{A}\stackrel{x}{\longrightarrow}\widetilde{B}\stackrel{y}{\longrightarrow}\widetilde{C}\stackrel{\widetilde{\omega}}\dashrightarrow.$$

Let $\Delta_{2}$, $\Delta$ be as above, then we have a sequence of maps $\Delta_{2}\stackrel{i}{\longrightarrow}\Delta\stackrel{\pi}{\longrightarrow}\Delta_{2}$ such that $\pi i=\id_{\Delta(2)}$ and $\pi $ is also a morphism of $\mathbb{F}$-triangles from $\tiny\begin{pmatrix} \widetilde{\omega}&0\\0&0 \end{pmatrix}$ to $\widetilde{\omega}$.

Let $\widetilde{\omega}\in\mathbb{F}(\widetilde{C},\widetilde{A})$ and $\widetilde{\omega'}\in\mathbb{F}(\widetilde{C'},\widetilde{A'})$ be any pair of $\mathbb{F}$-extensions, with

$$\Delta(1):~\widetilde{A}\stackrel{x}{\longrightarrow}\widetilde{B}\stackrel{y}{\longrightarrow}\widetilde{C}\stackrel{\widetilde{\omega}}\dashrightarrow$$
and
$$\Delta(2):~\widetilde{A'}\stackrel{x'}{\longrightarrow}\widetilde{B'}\stackrel{y'}{\longrightarrow}\widetilde{C'}\stackrel{\widetilde{\omega}}\dashrightarrow.$$
For any morphism $(a,c):\widetilde{\omega}\rightarrow\widetilde{\omega'}$ in $\mathbb{F}$-Ext$(\widetilde{\mathscr{C}})$, we have a sequence of maps $\Delta(1)\stackrel{i}{\longrightarrow}\Delta\stackrel{\pi}{\longrightarrow}\Delta(1)$ and $\Delta(2)\stackrel{i'}{\longrightarrow}\nabla\stackrel{\pi'}{\longrightarrow}\Delta(2)$ such that $\pi i=\id_{\Delta(1)}$ and $\pi' i'=\id_{\Delta(2)}$, where $\Delta$ and $\nabla$ are $\mathbb{E}$-triangles for $\delta$ and $\delta'$, respectively. The morphism $(a,c)$ induces a morphism $(i_{a'}a\pi_{a},i_{c'}c\pi_{c}):\delta\rightarrow\delta'$. We can apply $\rm (ET3)$ to extend the map $(i_{a'}a\pi_{a},i_{c'}c\pi_{c})$ to a morphism of  $\mathbb{E}$-triangles $\alpha:\Delta\rightarrow\nabla$. Then $\pi'\alpha i$ is a morphism from $ \Delta(1)$ to $\Delta(2)$. Thus $\mathfrak{c}$ is an additive realization of $\mathbb{F}$. In conclusion, we have the following

\begin{lemma}\label{2}
Let $\widetilde{\mathscr{C}},\mathbb{F},\mathfrak{c}$ be as above. Then $\mathfrak{c}$ is an additive realization for $\mathbb{F}$.
\end{lemma}

We show that the triplet $(\widetilde{\mathscr{C}},\mathbb{F},\mathfrak{c})$ is also compatible to $\rm (ET3)$.

\begin{proposition}\label{6}
Let $\widetilde{\omega}\in\mathbb{F}(\widetilde{C},\widetilde{A})$ and $\widetilde{\omega'}\in\mathbb{F}(\widetilde{C'},\widetilde{A'})$ be any pair of $\mathbb{F}$-extensions, realized as
$\widetilde{A}\stackrel{x}{\longrightarrow}\widetilde{B}\stackrel{y}{\longrightarrow}\widetilde{C}$ and  $\widetilde{A'}\stackrel{x'}{\longrightarrow}\widetilde{B'}\stackrel{y'}{\longrightarrow}\widetilde{C'}$, respectively. For any commutative square in $\widetilde{\mathscr{C}}$
$$\xymatrix{
 \Delta_{1}\ar[d]_-{(a,b)} &\widetilde{A} \ar[d]_-{a} \ar[r]^-{x} &\widetilde{ B} \ar[d]_-{b} \ar[r]^{y} & \widetilde{C}\ar@{-->}[r]^-{\widetilde{\omega}} &\\
 \Delta_{2}&\widetilde{A'}\ar[r]^-{x'} &\widetilde{B'}\ar[r]^-{y'} & \widetilde{C'}\ar@{-->}[r]^-{\widetilde{\omega'}}&}$$
there exists a morphism $(a,c)$: $\widetilde{\omega}\rightarrow\widetilde{\omega'}$ which is realized by $(a,b,c)$.
\end{proposition}
\begin{proof}
By the construction of $\mathbb{F}$-triangles, there exists a sequence of maps $\Delta_{1}\stackrel{i}{\longrightarrow}\Delta\stackrel{\pi}{\longrightarrow}\Delta_{1}$ and $\Delta_{2}\stackrel{i'}{\longrightarrow}\nabla\stackrel{\pi'}{\longrightarrow}\Delta_{2}$ such that $\pi i=\id_{\Delta_{1}}$ and $\pi' i'=\id_{\Delta_{2}}$, where $\Delta$ and $\nabla$ are $\mathbb{E}$-triangles in $\mathscr{C}$. The map $(a,b)$ induces a map $(i_{a}'a\pi_{a},i_{b}'b\pi_{b})$ from $\Delta$ to $\nabla$. Using $\rm (ET3)$, we complete the map to a morphism of $\mathbb{E}$-triangles $\varphi$: $\Delta\rightarrow\nabla$. Then $\pi'\varphi i$ is a morphism of $\mathbb{F}$-triangles from $\Delta_{1}$ to $\Delta_{2}$. This finishes the proof.
\end{proof}

We can check that $(\widetilde{\mathscr{C}},\mathbb{F},\mathfrak{c})$ is compatible to $\rm(ET3)^{op}$ in the same way.

\begin{proposition}\label{7}
The idempotent completion of a pre-extriangulated category has a natural structure of a pre-extriangulated category.
\end{proposition}
\begin{proof}
It follows from Lemmas \ref{1}, \ref{2} and Proposition \ref{6}.
\end{proof}

Let $(\mathscr{C},\mathbb{E},\mathfrak{s})$ be a pre-extriangulated category. We say an $\mathbb{E}$-triangle is {\em splitting} if it realizes a split $\mathbb{E}$-extension. The $\mathbb{E}$-triangle $A\stackrel{x}{\longrightarrow}B\stackrel{y}{\longrightarrow}C\stackrel{\delta}\dashrightarrow$ sometimes is denoted by the triplet $(x,y,\delta)$.
\begin{lemma}\label{11} Consider the following commutative diagram
$$\xymatrix{
  A \ar[d]_{x} \ar[r]^{u} & B \ar[d]_{y} \ar[r]^{v} & C \ar[d]_{z} \ar@{-->}[r]^{\delta} &  \\
  A' \ar[r]^{a} & B' \ar[r]^{b} & C' &  }$$
where the vertical morphisms are isomorphism. Then $(a,b,x_{\ast}(z^{-1})^{\ast}\delta)$ is an $\mathbb{E}$-triangle.
\end{lemma}
\begin{proof} First of all, by \cite[Proposition 3.7]{Na}, we know that $A'\stackrel{ux^{-1}}{\longrightarrow}B\stackrel{zv}{\longrightarrow}C'\stackrel{x_{\ast}(z^{-1})^{\ast}\delta}\dashrightarrow$ is an $\mathbb{E}$-triangle in $\mathscr{C}$. Note that $a=yux^{-1}$ and $b=zvy^{-1}$. It follows that  $A'\stackrel{a}{\longrightarrow}B'\stackrel{b}{\longrightarrow}C'\stackrel{x_{\ast}(z^{-1})^{\ast}\delta}\dashrightarrow$ is an $\mathbb{E}$-triangle in $\mathscr{C}$.

\end{proof}

\begin{lemma}\label{9}
Consider the $\mathbb{E}$-triangle
$$\Delta:~~~~~~~~~~~~~A\oplus X\stackrel{\tiny\begin{pmatrix} a&u\\0&x\end{pmatrix}}{\longrightarrow}B\oplus Y\stackrel{\tiny\begin{pmatrix} b&v\\0&y\end{pmatrix}}{\longrightarrow}C\oplus Z\stackrel{\tiny\begin{pmatrix} c&w\\0&z\end{pmatrix}}\dashrightarrow$$
in a pre-extriangulated category. If $(a,b,c)$ is a direct sum of two split $\mathbb{E}$-triangles $(\id_{A},0,0)$ and $(0,\id_{C},0)$. Then $(x,y,z)$ is an  $\mathbb{E}$-triangle. Dually, If $(x,y,z)$ is a direct sum of two split $\mathbb{E}$-triangles $(\id_{X},0,0)$ and $(0,\id_{Z},0)$. Then $(a,b,c)$ is an  $\mathbb{E}$-triangle.
\end{lemma}
\begin{proof}
Suppose that $(a,b,c)$ is a direct sum of two split $\mathbb{E}$-triangles $(\id_{A},0,0)$ and $(0,\id_{C},0)$. Then $\Delta$ can be rephrased as follows:
$$\Delta:A\oplus X\stackrel{\tiny\begin{pmatrix} 1&u_{1}\\0&u_{2}\\0&x \end{pmatrix}}{\longrightarrow}A\oplus C\oplus Y\stackrel{\tiny\begin{pmatrix} 0&1&v\\0&0&y\end{pmatrix}}{\longrightarrow}C\oplus Z\stackrel{\tiny\begin{pmatrix} 0&w\\0&z\end{pmatrix}}\dashrightarrow.$$
Consider the  following commutative diagram
$$\xymatrix{
\Delta_{1} \ar[d]: &A\oplus X \ar[d]_-{m} \ar[r]^-{\tiny\begin{pmatrix} 1&u_{1}\\0&u_{2}\\0&x \end{pmatrix}} & A\oplus C\oplus Y \ar[d]_-{n} \ar[r]^-{\tiny\begin{pmatrix} 0&1&v\\0&0&y\end{pmatrix}} & C\oplus Z \ar[d]_-{l}  &  \\
\Delta_{2}  & A\oplus X \ar[r]^-{\tiny\begin{pmatrix} 1&0\\0&0\\0&x \end{pmatrix}} & A\oplus C\oplus Y \ar[r]^-{\tiny\begin{pmatrix} 0&1&0\\0&0&y \end{pmatrix}} &  C\oplus Z  &   }
$$
where $m=\tiny\begin{pmatrix} 1&u_{1}\\0&1\end{pmatrix},~n=\tiny\begin{pmatrix} 1&0&0\\0&1&v\\0&0&1 \end{pmatrix},~l=\tiny\begin{pmatrix} 1&0\\0&1\end{pmatrix}$. It is easy to check that the three morphisms yield an isomorphism of $\Delta_{1}$ and $\Delta_{2}$.  It is a straightforward verification that $m_{\ast}\tiny\begin{pmatrix} 0&w\\0&z\end{pmatrix}=l^{\ast}\tiny\begin{pmatrix} 0&0\\0&z\end{pmatrix}$. By Lemma \ref{11}, we know that $\Delta_{2}$ is a conflation and realizes $\tiny\begin{pmatrix} 0&0\\0&z\end{pmatrix}$. By Lemma \ref{second}, we obtain that $(x,y,z)$ is an  $\mathbb{E}$-triangle. For the dual case, it is proved similarly.
\end{proof}

Now we are ready to prove Theorem \ref{main}.\\


\textbf{{Proof of Theorem \ref{main}.}}

By Proposition \ref{7}, it is enough to show that $\widetilde{\mathscr{C}}$ is compatible to $\rm(ET4)$ and $\rm(ET4)^{op}$.
Let $\widetilde{\omega}\in\mathbb{F}(\widetilde{D},\widetilde{A})$ and $\widetilde{\omega'}\in\mathbb{F}(\widetilde{F},\widetilde{B})$ be two $\mathbb{F}$-extensions respectively realized by
\begin{equation}\label{3.5}
\widetilde{A}\stackrel{f}{\longrightarrow}\widetilde{B}\stackrel{f'}{\longrightarrow}\widetilde{D}\stackrel{\widetilde{\omega}}\dashrightarrow
\end{equation}
and
\begin{equation}\label{3.6}
\widetilde{B}\stackrel{g}{\longrightarrow}\widetilde{C}\stackrel{g'}{\longrightarrow}\widetilde{F}\stackrel{\widetilde{\omega'}}\dashrightarrow.
\end{equation}
There exist $\widetilde{X},\widetilde{Y}$ and $\widetilde{Z}$ such that $\widetilde{A}\oplus \widetilde{X}$, $\widetilde{D}\oplus \widetilde{Y}$, $\widetilde{F}\oplus \widetilde{Z}\in\mathscr{C}$. Clearly,
\begin{equation}\label{3.7}
\widetilde{X}\stackrel{\id_{\widetilde{X}}}{\longrightarrow}\widetilde{X}\stackrel{}{\longrightarrow}0\stackrel{0}\dashrightarrow
\end{equation}
and
\begin{equation}\label{3.8}
0\stackrel{}{\longrightarrow}\widetilde{Y}\stackrel{\id_{\widetilde{Y}}}{\longrightarrow}\widetilde{Y}\stackrel{0}\dashrightarrow
\end{equation}
are split $\mathbb{F}$-triangles. Considering the direct sum of $(\ref{3.5})$, $(\ref{3.7})$ and $(\ref{3.8})$, we obtain the following $\mathbb{F}$-triangle:
\begin{equation}\label{3.9}
\widetilde{A}\oplus \widetilde{X}\stackrel{\tiny\begin{pmatrix} f&0\\0&1\\0&0 \end{pmatrix}}{\longrightarrow}\widetilde{B}\oplus \widetilde{X}\oplus \widetilde{Y}\stackrel{\tiny\begin{pmatrix} f'&0&0\\0&0&1 \end{pmatrix}}{\longrightarrow}\widetilde{D}\oplus \widetilde{Y}\stackrel{\widetilde{\omega_{1}}}\dashrightarrow
\end{equation}
where $\widetilde{\omega_{1}}=\tiny\begin{pmatrix} \widetilde{\omega}&0\\0&0\end{pmatrix}$. Observe that $\widetilde{\omega_{1}}\in\mathbb{E}(\widetilde{D}\oplus \widetilde{Y},\widetilde{A}\oplus \widetilde{X})$ fits into an $\mathbb{E}$-triangle of $\mathscr{C}$ which is, via$\iota$, an $\mathbb{F}$-triangle of $\widetilde{\mathscr{C}}$. These two triangles are isomorphic since $\widetilde{\mathscr{C}}$ is a pre-extriangulated category. Then $(\ref{3.9})$ is an $\mathbb{E}$-triangle for $\widetilde{\omega_{1}}$ in $\mathscr{C}$.

Similarly, we consider the following  $\mathbb{F}$-triangle in $\widetilde{\mathscr{C}}$ by the direct sum of $(\ref{3.6})$, $(\id_{\widetilde{Y}},0,0)$ and $(0,\id_{\widetilde{Z}},0)$
\begin{equation}\label{3.10}
\widetilde{B}\oplus \widetilde{Y}\stackrel{\tiny\begin{pmatrix} g&0\\0&1\\0&0 \end{pmatrix}}{\longrightarrow}\widetilde{C}\oplus \widetilde{Y}\oplus \widetilde{Z}\stackrel{\tiny\begin{pmatrix} g'&0&0\\0&0&1 \end{pmatrix}}{\longrightarrow}\widetilde{F}\oplus \widetilde{Z}\stackrel{\widetilde{\xi}}\dashrightarrow
\end{equation}
where $\widetilde{\xi}=\tiny\begin{pmatrix}\widetilde{ \omega'}&0\\0&0\end{pmatrix}$. Adding $(\ref{3.7})$ to $(\ref{3.10})$, we have
\begin{equation}\label{3.11}
\widetilde{B}\oplus \widetilde{X}\oplus \widetilde{Y}\stackrel{h}{\longrightarrow}\widetilde{C}\oplus \widetilde{X}\oplus \widetilde{Y}\oplus \widetilde{Z } \stackrel{h'}{\longrightarrow}\widetilde{F}\oplus \widetilde{Z}\stackrel{\widetilde{\omega_{2}}}\dashrightarrow,
\end{equation}
where $h=\tiny\begin{pmatrix} g&0& 0\\0&1&0\\0&0 &1\\0&0& 0 \end{pmatrix}$, $h'=\tiny\begin{pmatrix} g'&0&0&0\\0&0&0&1 \end{pmatrix}$ and $\widetilde{\omega_{2}}=\tiny\begin{pmatrix} \widetilde{\omega'}&0\\0&0\\0&0\end{pmatrix}$.

By $\rm(ET4)$, we have the following  commutative diagram  since $(\ref{3.9})$ and $(\ref{3.11})$ are $\mathbb{E}$-triangles in $\mathscr{C}$.
$$\xymatrix{
  \widetilde{A}\oplus \widetilde{X}\ar@{=}[d]\ar[r]^-{\tiny\begin{pmatrix} f&0\\0&1\\0&0 \end{pmatrix}} &\widetilde{B}\oplus \widetilde{X}\oplus \widetilde{Y}\ar[d]^-{h} \ar[r]^-{\tiny\begin{pmatrix} f'&0&0\\0&0&1\end{pmatrix}} & \widetilde{D}\oplus \widetilde{Y}\ar[d]^-{\alpha=(\alpha_{1}~\alpha_{2})} \ar@{-->}[r]^-{\widetilde{\omega_{1}}}& \\
  \widetilde{A}\oplus \widetilde{X}\ar[r]_-{\tiny\begin{pmatrix} gf&0\\0&1\\0&0\\0&0 \end{pmatrix}} & \widetilde{C}\oplus\widetilde{ X}\oplus \widetilde{Y}\oplus \widetilde{Z}\ar[d]^-{h'} \ar[r]_-{i=(i_{1}i_{2}i_{3}i_{4})} & H\ar[d]^-{\beta=\tiny\begin{pmatrix}\beta_{1}\\\beta_{2} \end{pmatrix}}\ar@{-->}[r]^-{\widetilde{\omega_{3}}} &\\
   & \widetilde{F}\oplus \widetilde{Z}\ar@{=}[r]\ar@{-->}[d]^-{\widetilde{\omega_{2}}} & \widetilde{F}\oplus \widetilde{Z}\ar@{-->}[d]^-{\widetilde{\omega_{4}}}  &\\
   \ & & & } $$

By some direct calculations, we obtain
$$\begin{cases}
(\alpha_{1}f',0,\alpha_{2})=(i_{1}g,i_{2},i_{3})& \text{}\\\beta_{1}i_{1}=g',~~~\beta_{2}i_{4}=1
& \text{}\\(\alpha_{1}~\alpha_{2})^{\bullet}\widetilde{\omega_{3}}=(\alpha_{1}~\alpha_{2})^{\bullet}(\widetilde{\omega_{31}}~\widetilde{\omega_{32}})^{T}=\tiny\begin{pmatrix} \alpha_{1}^{\bullet}\widetilde{\omega_{31}}&0\\0&\alpha_{2}^{\bullet}\widetilde{\omega_{32}}\end{pmatrix}=\widetilde{\omega_{1}}\\\tiny\begin{pmatrix} f&0\\0&1\\0&0\end{pmatrix}_{\bullet}(\widetilde{\omega_{31}}~\widetilde{\omega_{32}})^{T}=\tiny\begin{pmatrix} f_{\bullet}\widetilde{\omega_{31}}\\\widetilde{\omega_{32}}\\0\end{pmatrix}=\beta^{\bullet}\widetilde{\omega_{2}}=\tiny\begin{pmatrix} \beta_{1}^{\bullet}\widetilde{\omega'}\\0\\0\end{pmatrix}
\\\tiny\begin{pmatrix} f'&0&0\\0&0&1 \end{pmatrix}_{\bullet}\widetilde{\omega_{2}}=\tiny\begin{pmatrix} f'_{\bullet}\widetilde{\omega'}&0\\0&0 \end{pmatrix}=\widetilde{\omega_{4}}=\tiny\begin{pmatrix} \widetilde{\omega_{41}}&\widetilde{\omega_{42}}\\ \widetilde{\omega_{43}}&\widetilde{\omega_{44}} \end{pmatrix}.
\text{ }
\end{cases}
$$
Thus $\widetilde{\omega_{4}}=\tiny\begin{pmatrix} \widetilde{\omega_{41}}&0\\0&0\end{pmatrix}$, where $\widetilde{\omega_{41}}\in\mathbb{F}(\widetilde{F},\widetilde{D})$.
We have $\mathbb{F}$-triangles
\begin{equation}\label{3.12}
\widetilde{D}\stackrel{p_{1}}{\longrightarrow}\widetilde{K}\stackrel{p_{2}}{\longrightarrow}\widetilde{F}\stackrel{\widetilde{\omega_{41}}}\dashrightarrow
\end{equation}
and
\begin{equation}\label{3.13}
\widetilde{Y}\stackrel{\tiny\begin{pmatrix} 1\\0\end{pmatrix}}{\longrightarrow}\widetilde{Y}\oplus \widetilde{Z}\stackrel{\tiny\begin{pmatrix} 0&1\end{pmatrix}}{\longrightarrow}\widetilde{Z}\stackrel{0}\dashrightarrow.
\end{equation}
Consider the direct sum of $(\ref{3.12})$ and $(\ref{3.13})$, we have an isomorphism $s=(s_{1},s_{2},s_{3})^{T}:H\rightarrow \widetilde{K}\oplus \widetilde{Y}\oplus \widetilde{Z}$ with its inverse $s'=(s_{1}',s_{2}',s_{3}')$ such that

$$s\alpha=\tiny\begin{pmatrix} s_{1}\alpha_{1}& s_{2}\alpha_{2}\\s_{2}\alpha_{1}&s_{2}\alpha_{2}\\s_{3}\alpha_{1}& s_{3}\alpha_{2}\end{pmatrix}=\tiny\begin{pmatrix}p_{1}& 0\\0&1\\0& 0\end{pmatrix}$$
$$\beta s'=\tiny\begin{pmatrix} \beta_{1}s_{1}'&\beta_{1}s_{2}'&\beta_{1}s_{3}'&\\\beta_{2}s_{1}'&\beta_{2}s_{2}'&\beta_{2}s_{3}'&\end{pmatrix}=\tiny\begin{pmatrix}p_{2}&0&0\\0&0&1\end{pmatrix}.$$

Let us modify the diagram above. Firstly, we consider the following commutative diagram:
$$
\xymatrix{
  \widetilde{A}\oplus \widetilde{X} \ar@{=}[d] \ar[r] & \widetilde{C}\oplus \widetilde{X}\oplus \widetilde{Y}\oplus \widetilde{Z}\ar@{=}[d] \ar[r]^-{i} & H \ar[d]_-{s} \ar@{-->}[r]^-{\widetilde{\omega_{3}}} &  \\
   \widetilde{A}\oplus \widetilde{X}\ar[r]& \widetilde{C}\oplus \widetilde{X}\oplus \widetilde{Y}\oplus \widetilde{Z}  \ar[r]^-{si} & \widetilde{K}\oplus \widetilde{Y}\oplus \widetilde{Z} \ar@{-->}[r]^-{\widetilde{\omega_{5}}} &   }$$
where $\widetilde{\omega_{5}}={s'}^{\bullet}\widetilde{\omega_{3}}=\tiny\begin{pmatrix} s_{1}'^{\bullet}\widetilde{\omega_{31}}& s_{2}'^{\bullet}\widetilde{\omega_{31}}& s_{3}'^{\bullet}\widetilde{\omega_{31}}\\0&0&0\end{pmatrix}$. Obviously, the second row is an $\mathbb{E}$-triangle for $\omega_{5}\in\mathbb{E}(\widetilde{K}\oplus \widetilde{Y}\oplus \widetilde{Z},\widetilde{A}\oplus \widetilde{X})$ by the isomorphism. Now we have a new commutative diagram in $\widetilde{\mathscr{C}}$  as follows:
$$\xymatrix{
  \widetilde{A}\oplus \widetilde{X}\ar@{=}[d]\ar[r]^-{\tiny\begin{pmatrix} f&0\\0&1\\0&0 \end{pmatrix}~~~~~} &\widetilde{B}\oplus \widetilde{X}\oplus \widetilde{Y}\ar[d]^-{h} \ar[r]^-{\tiny\begin{pmatrix} f'&0&0\\0&0&1\end{pmatrix}} & \widetilde{D}\oplus \widetilde{Y}\ar[d]^-{s\alpha} \ar@{-->}[r]^-{\widetilde{\omega_{1}}}&\\
  \widetilde{A}\oplus \widetilde{X}\ar[r] & \widetilde{C}\oplus \widetilde{X}\oplus \widetilde{Y}\oplus \widetilde{Z}\ar[d]^-{h'} \ar[r]^-{si} & \widetilde{K}\oplus \widetilde{Y}\oplus \widetilde{Z} \ar[d]^-{\beta s'}\ar@{-->}[r]^-{\widetilde{\omega_{5}}} &\\
   & \widetilde{F}\oplus \widetilde{Z}\ar@{=}[r]\ar@{-->}[d]^-{\widetilde{\omega_{2}}} & \widetilde{F}\oplus \widetilde{Z}\ar@{-->}[d]^-{\widetilde{\omega_{4}}}  &\\
   \ & & & } $$
Calculating again, we have
$$
\begin{cases}
s_{1}i_{1}g=p_{1}f', s_{1}\alpha_{2}=s_{3}\alpha_{2}=0, s_{2}\alpha_{2}=1 &\text{}\\
p_{2}s_{1}i_{1}=g', s_{3}i_{1}=0, s_{3}i_{4}=1, s_{3}i_{1}=0
& \text{}\\(s\alpha)^{\bullet}\widetilde{\omega_{5}}=\widetilde{\omega_{1}},~\widetilde{\omega}=p_{1}^{\bullet}s_{1}'^{\bullet}\widetilde{\omega_{31}},~s_{2}'^{\bullet}\widetilde{\omega_{31}}=0
\\(\beta s')^{\bullet}\widetilde{\omega_{2}}={\tiny\begin{pmatrix} f&0\\0&1\\0&0 \end{pmatrix}}_{\bullet}\widetilde{\omega_{5}},~p_{2}^{\bullet}\widetilde{\omega'}=f_{\bullet}s_{1}'^{\bullet}\widetilde{\omega_{31}}.
\text{ }
\end{cases}$$
Obviously, we have $si=\tiny\begin{pmatrix}s_{1}i_{1}&0&0&s_{1}i_{4}\\s_{2}i_{1}&0&1&s_{2}i_{4}\\0&0&0&0\end{pmatrix}$ and $\widetilde{\omega_{5}}=\tiny\begin{pmatrix} s_{1}'^{\bullet}\widetilde{\omega_{31}}& 0& s_{3}'^{\bullet}\widetilde{\omega_{31}}\\0&0&0\end{pmatrix}$.
Putting all these together, we obtain the following commutative diagram in $\widetilde{\mathscr{C}}$:
$$\xymatrix{
  \widetilde{A}\ar@{=}[d]\ar[r]^-{f} &\widetilde{B}\ar[d]^-{g} \ar[r]^-{f'} & \widetilde{D}\ar[d]^-{p_{1}} \ar@{-->}[r]^-{\widetilde{\omega}}& \\
  \widetilde{A}\ar[r]^-{gf} & \widetilde{C}\ar[d]^-{g'} \ar[r]^-{s_{1}i_{1}} & \widetilde{K}\ar[d]^-{p_{2}} \ar@{-->}[r]^-{{s_{1}'}^{\bullet}\widetilde{\omega_{31}}} &\\
   & \widetilde{F}\ar@{=}[r]\ar@{-->}[d]^-{\widetilde{\omega'}} & \widetilde{F}\ar@{-->}[d]^-{\widetilde{\omega_{41}}}  &\\
   \ & & & } $$
which is compatible to $\rm(ET4)$. It should be noted that
\begin{equation}\label{3.14}\widetilde{A}\stackrel{gf}{\longrightarrow}\widetilde{C}\stackrel{s_{1}i_{1}}{\longrightarrow}\widetilde{K}\stackrel{{s_{1}'}^{\bullet}\widetilde{\omega_{31}}}\dashrightarrow\end{equation}
is an $\mathbb{F}$-triangle in $\widetilde{\mathscr{C}}$. Indeed, we have the following morphisms of $\mathbb{F}$-triangles
$$\xymatrix{
 \bigtriangleup(1)\ar@{^(->}[d]& \widetilde{X} \ar@{^(->}[d] \ar[r]^-{1\choose0} &\widetilde{X}\oplus \widetilde{Y}\ar@{^(->}[d] \ar[r]^-{(0\quad1)} &\widetilde{Y } \ar@{^(->}[d]\ar@{-->}[r]^-{0} &  \\
 \bigtriangleup(2)\ar@{^(->}[d] &\widetilde{X}\oplus \widetilde{A} \ar@{^(->}[d] \ar[r]^-{\tiny\begin{pmatrix} 1&0\\0&0\\0&gf \end{pmatrix}} & \widetilde{X}\oplus \widetilde{Y}\oplus \widetilde{C} \ar@{^(->}[d] \ar[r]^-{\tiny\begin{pmatrix} 0&1&s_{2}i_{1}\\0&0&s_{1}i_{1} \end{pmatrix}} & \widetilde{Y}\oplus \widetilde{K} \ar@{^(->}[d] \ar@{-->}[r]^-{\tiny\begin{pmatrix} 0&0\\0&s_{1}'^{\bullet}\widetilde{\omega_{31}}\end{pmatrix}} &  \\
  \bigtriangleup(3)&\widetilde{A}\oplus \widetilde{X} \ar[r]^-{\sigma} &\widetilde{C}\oplus \widetilde{X}\oplus \widetilde{Y}\oplus \widetilde{Z}\ar[r]^-{\mu} & \widetilde{K}\oplus \widetilde{Y}\oplus \widetilde{Z} \ar@{-->}[r]^-{\widetilde{\omega_{5}}} &   }$$
where $\sigma=\tiny\begin{pmatrix} gf&0\\0&1\\0&0\\0&0 \end{pmatrix}$, $\mu=\tiny\begin{pmatrix} s_{1}i_{1}&0&0&s_{1}i_{4}\\s_{2}i_{1}&0&1 &s_{2}i_{4}\\0&0&0&1 \end{pmatrix}$, $\widetilde{\omega_{5}}=\tiny\begin{pmatrix} s_{1}'^{\bullet}\widetilde{\omega_{31}}&0&s_{3}'^{\bullet}\widetilde{\omega_{31}}\\0&0&0\end{pmatrix}$, and all $\hookrightarrow$ are the canonical embeddings. By Lemma \ref{9}, we know that $\bigtriangleup(2)$ is an $\mathbb{F}$-triangle in $\widetilde{\mathscr{C}}$. It follows that $(\ref{3.14})$ is an $\mathbb{F}$-triangle since $\bigtriangleup(1)$ is a splitting $\mathbb{F}$-triangle.
\fin\\

Recall that an additive functor between two extriangulated categories is {\em exact} if the functor sends conflations to conflations.

\begin{corollary}Let $\mathscr{C}$ be an extriangulated category. Then its idempotent completion  $\widetilde{\mathscr{C}}$ admits a smallest structure of an extriangulated  category such that the canonical functor $i: \mathscr{C}\rightarrow\widetilde{ \mathscr{C}}$ becomes exact. Moreover, if $\widetilde{\mathscr{C}}$ is endowed with this structure, then for each idempotent complete extriangulated category $\mathscr{D}$, the functor$\iota$ induces an equivalence
$$\xymatrix{\Hom_{exact}(\widetilde{ \mathscr{C}},\mathscr{D})\ar[r]^-{\sim}& \Hom_{exact}(\mathscr{C},\mathscr{D})},$$
where $\Hom_{exact}$ denotes the category of exact functors between two extriangulated categories.
\end{corollary}
\begin{proof}
Suppose that $\widetilde{\mathscr{C}}$ has another extriangulated structure such that $i: \mathscr{C}\rightarrow\widetilde{ \mathscr{C}}$ becomes exact, then it contains all $\mathbb{F}$-triangles by Lemma \ref{second} and Lemma \ref{2}. Then the rest follows from Lemma \ref{third} and Lemma \ref{second}.
\end{proof}

\begin{remark} Recall that each exact structure or triangulated structure forms an extriangulated structure in the sense of \cite[Example 2.13]{Na} and \cite[Proposition 3.22]{Na}, respectively. Let $(\mathscr{C},\mathbb{E},\mathfrak{s})$ be an extriangulated category which is neither exact nor triangulated. By Theorem \ref{main}, we know that the idempotent completion $(\widetilde{\mathscr{C}},\mathbb{F},\mathfrak{c})$ of $(\mathscr{C},\mathbb{E},\mathfrak{s})$ also carries an extriangulated structure. We claim that this extriangulated structure is neither exact nor triangulated.

Assume that $(\widetilde{\mathscr{C}},\mathbb{F},\mathfrak{c})$ is an exact category. By [8, Corollary 3.18], we know that any inflation in $\widetilde{\mathscr{C}}$ is monomorphic, and any deflation in $\widetilde{\mathscr{C}}$ is epimorphic. It follows that any inflation in $\mathscr{C}$ is monomorphic, and any deflation in $\mathscr{C}$ is epimorphic. Thus, by again [8, Corollary 3.18], $(\mathscr{C},\mathbb{E},\mathfrak{s})$ is exact, this is a contradiction. Hence, $(\widetilde{\mathscr{C}},\mathbb{F},\mathfrak{c})$ is not exact. Next, we prove $(\widetilde{\mathscr{C}},\mathbb{F},\mathfrak{c})$ is not triangulated.

First of all, we observe that each project object in $\mathscr{C}$ is also projective in $\widetilde{\mathscr{C}}$. Indeed, let $P$ be a projective object in $\mathscr{C}$, i.e., $\mathbb{E}(P,B)=0$ for any $B\in\mathscr{C}$. For any $\widetilde{B}=(B,e)\in\widetilde{\mathscr{C}}$, take $\widetilde{B}'=(B,\id_{B}-e)\in\widetilde{\mathscr{C}}$, then $\mathbb{F}(P,\widetilde{B})\oplus\mathbb{F}(P,\widetilde{B}')\cong\mathbb{F}(P,\widetilde{B}\oplus\widetilde{B}')\cong\mathbb{E}(P,B)=0$, thus $\mathbb{F}(P,\widetilde{B})=0$, that is, $P$ is projective in $\widetilde{\mathscr{C}}$. Similarly, we can prove each injective object in $\mathscr{C}$ is also injective in $\widetilde{\mathscr{C}}$.

Suppose that $(\widetilde{\mathscr{C}},\mathbb{F},\mathfrak{c})$ is a triangulated category. By \cite[Corollary 7.6]{Na}, $(\widetilde{\mathscr{C}},\mathbb{F},\mathfrak{c})$ is Frobenius with $\mathcal{P}=\mathcal{I}=0$. According to the observations above, we obtain that $\mathscr{C}$ has no nonzero projective or injective objects.
Now we prove $\mathscr{C}$ has enough projectives and injectives. For any non-zero object $C=(C,\id_{C})\in\widetilde{\mathscr{C}}$, there exists an $\mathbb{F}$-triangle
\begin{equation}\label{j1}
\widetilde{A}=(A,e)\stackrel{}{\longrightarrow}0\stackrel{}{\longrightarrow}{(C,\id_{C})}\stackrel{\widetilde{\omega}}\dashrightarrow
\end{equation}
since $\widetilde{\mathscr{C}}$ has enough projectives and $\mathcal{P}=0$ in $\widetilde{\mathscr{C}}$. Take $\widetilde{A}'=(A,\id_{A}-e)$, we consider the following $\mathbb{F}$-triangle obtained by the direct sum of (\ref{j1}) and the split $\mathbb{F}$-triangle $(\id_{\widetilde{A}'},0,0)$
\begin{equation}\label{j2}(A,\id_{A})\stackrel{}{\longrightarrow}\widetilde{A}'\stackrel{}{\longrightarrow}(C,\id_{C})\stackrel{\tiny\begin{pmatrix} \widetilde{\omega}&0\\0&0 \end{pmatrix}}\dashrightarrow.
\end{equation}
Then $\widetilde{A}'\in\mathscr{C}$, since $\begin{pmatrix} \widetilde{\omega}&0\\0&0 \end{pmatrix}\in\mathbb{E}(C,A)$. It follows that either $e=0$ or $e=\id_{A}$. If $e=0$, then $\widetilde{A}\cong0$, thus by (\ref{j1}), $C=0$, this contradicts the hypothesis $C$ is nonzero. Hence, $e=\id_{A}$. Thus, the $\mathbb{F}$-triangle (\ref{j1}) is actually an $\mathbb{E}$-triangle ${A}\stackrel{}{\longrightarrow}0\stackrel{}{\longrightarrow}{C}\stackrel{\widetilde{\omega}}\dashrightarrow$, so $\mathscr{C}$ has enough projectives. Similarly, it is proved that $\mathscr{C}$ has enough injectives.
Thus, by \cite[Corollary 7.6]{Na}, $\mathscr{C}$ is triangulated, this is a contradiction.
Hence, $(\widetilde{\mathscr{C}},\mathbb{F},\mathfrak{c})$ is not a triangulated category.

\end{remark}

\section{Applications}

In this section, we give two applications about the idempotent completion of extriangulated categories.

\subsection{Cotorsion pairs}
In this subsection,  we focus our attention on cotorsion pairs in the idempotent completion of an extriangulated category. Let $(\mathscr{C},\mathbb{E},\mathfrak{s})$ be an extriangualted category with the idempotent completion $(\widetilde{\mathscr{C}},\mathbb{F},\mathfrak{c})$.

Recall that the pair $(\mathcal{T},\mathcal{F})$ is called a {\em cotorsion pair} in $\mathscr{C}$ if it satisfies the following conditions:

$\bullet$ $\mathbb{E}(\mathcal{T},\mathcal{F})=0.$

$\bullet$ For any $C\in\mathscr{C}$, there exists a conflation $F\longrightarrow T\longrightarrow C$ satisfying $F\in\mathcal{F}$, $T\in\mathcal{T}$.

$\bullet$ For any $C\in\mathscr{C}$, there exists a conflation $C\longrightarrow F'\longrightarrow T'$ satisfying $F'\in\mathcal{F}$, $T'\in\mathcal{T}$.

Obviously, the cotorsion pair induces a pair of subcategories $(\widetilde{\mathcal{T}},\widetilde{\mathcal{F}})$ in $\widetilde{\mathscr{C}}$, where
$$\widetilde{\mathcal{T}}=\{(T,t)\mid T\in\mathcal{T},~t\in \End(T)~is~an~idempotent\}$$
and
$$\widetilde{\mathcal{F}}=\{(F,f)\mid F\in\mathcal{F},~f\in \End(F)~is~an~idempotent\}.$$

\begin{theorem}\label{4.1}
Let $\mathscr{C}$ be an extriangulated category. Then any cotorsion pair $(\mathcal{T},\mathcal{F})$ in $\mathscr{C}$ which satisfies $\mathscr{C}(\mathcal{T},\mathcal{F})=0$ induces a cotorsion pair $(\widetilde{\mathcal{T}},\widetilde{\mathcal{F}})$ in $\widetilde{\mathscr{C}}$.
\end{theorem}
\begin{proof}
We have to verify the three conditions in the definition of cotorsion pairs.

$(1)$ For any two objects $\widetilde{T}=(T,t)\in\widetilde{\mathcal{T}}$ and $\widetilde{F}=(F,f)\in\widetilde{\mathcal{F}}$. Clearly, $\mathbb{F}(\widetilde{T},\widetilde{F})=0$ since $\mathbb{E}(T,F)=0$.

$(2)$ For any $\widetilde{C}=(C,e)\in\widetilde{\mathscr{C}}$ with $C\in\mathscr{C}$. Then there exists an $\mathbb{E}$-triangle in $\mathscr{C}$
$$\Delta:F\stackrel{x}\longrightarrow T\stackrel{y}\longrightarrow C\stackrel{\delta}\dashrightarrow$$
satisfying $F\in\mathcal{F}$, $T\in\mathcal{T}$.  Set $C_{1}=(C,e)$, $C_{2}=(C,\id_{C}-e)$, then we have $(C,\id_{C})\cong C_{1}\oplus C_{2}$. For the idempotent $e\in \mathscr{C}(C,C)$, since $\mathbb{E}({T},{F})=0$, we obtain a morphism
of $\mathbb{E}$-triangles:
$$\xymatrix{
  F  \ar[r]^-{x}\ar[d]_-{f} & T  \ar[r]^-{y} \ar[d]_-{t}& C \ar[d]_-{e} \ar@{-->}[r]^-{\delta} &  \\
  F \ar[r]^-{x} & T \ar[r]^-{y} & C \ar@{-->}[r]^-{\delta} &   .}
$$
Since $yt=ey$, we have $yt^2=eyt=e^2y=ey$. Thus there exists $h:F\rightarrow F$ such that $(h,t^{2},e)$ is also a morphism of $\mathbb{E}$-triangles from $\Delta$ to $\Delta$.
It follows that $h=f$ since $h^{\ast}\delta=e_{\ast}\delta=f^{\ast}\delta$. On the one hand, there exists a morphism $s:$ $T\rightarrow F$ such that $xs=t^{2}-t$ since $y(t^{2}-t)=0$. Hence $(t^{2}-t)^{2}=(t^{2}-t)xs=(t^{2}x-tx)s=0$.  On the other hand, we set $r=t+u-2tu$ where $u=t^{2}-t$. Noting that $u^2=0$ and $ut=tu$, we obtain that $r^{2}=t^{2}+2tu-4t^{2}u$ and then replacing $t^{2}$ by $u+t$, we have $r^{2}=t+u-2tu$. It is easy to see that $r$ makes the right hand square commutative. By Lemma \ref{first}, there exists an idempotent $g: F\rightarrow F$ such that $(g,r,e)$ is a morphism of $\mathbb{E}$-triangles  from $\Delta$ to $\Delta$. Note that $F\simeq(F,g)\oplus(F,\id_{F}-g)=:F_{1}\oplus F_{2}$, $T\simeq(T,r)\oplus(T,\id_{T}-r)=:T_{1}\oplus T_{2}$ and we have the following commutative diagram
$$\xymatrix{
 (F,\id_{F})  \ar[r]^{x}\ar[d]^{\cong} & (T,\id_{T})  \ar[r]^{y} \ar[d]^{\cong}& (C,\id_{C}) \ar[d]^{\cong} \ar@{-->}[r]^{\widetilde{\delta}} &  \\
  F_{1}\oplus F_{2}\ar[r]^{v} & T_{1}\oplus T_{2} \ar[r]^{w} & C_{1}\oplus C_{2}\ar@{-->}[r]^{\widetilde{\delta'}} &    }
  $$
where $v=\tiny\begin{pmatrix} xg&0\\0&x(\id_{F}-g)\end{pmatrix}$, $w=\tiny\begin{pmatrix} yr&0\\0&y(\id_{T}-
r)\end{pmatrix}$. In addition, by Lemma \ref{11}, we know that $\widetilde{\delta'}=\tiny\begin{pmatrix} g_{\bullet}e^{\bullet}\widetilde{\delta}&0\\0&(\id_{F}-g)_{\bullet}(\id_{C}-e)^{\bullet}\widetilde{\delta}\end{pmatrix}$. This shows that

$$F_{1}\stackrel{xf}\longrightarrow T_{1}\stackrel{yt}\longrightarrow C\stackrel{g_{\bullet}e^{\bullet}\widetilde{\delta}}\dashrightarrow$$
is an $\mathbb{F}$-triangle in $\widetilde{\mathscr{C}}$. Similarly, we can prove the third condition in the definition of a cotorsion pair. \end{proof}

It should be noted that if $\mathscr{C}$ is a triangulated category, then $(\mathcal{T},\mathcal{F})$ is a cotorsion pair if and only if $(\mathcal{T}[-1],\mathcal{F})$ is a torsion pair in \cite[Definition 2.2]{Iy2}. In particular, we have the following
\begin{corollary}
Let $\mathscr{C}$ be a triangulated category. Then any torsion pair $(\mathcal{T},\mathcal{F})$ in $\mathscr{C}$ which satisfies $\mathcal{F}[-1]\subseteq\mathcal{F}$ or $\mathcal{T}[1]\subseteq\mathcal{T}$ induces a torsion pair in $\widetilde{\mathscr{C}}$.
\end{corollary}

\subsection{Recollements}

In this subsection, we introduce the notion of right (left) recollements of extriangulated categories, and prove that
the idempotent completion of a right (resp. left) recollement of extriangulated categories is still a right ( resp. left) recollement, which generalizes \cite[Theorem 4]{CT}.

Let $\mathscr{C}$ and $\mathscr{D}$ be additive categories, and $F:\mathscr{C}\to \mathscr{D}$ an additive functor. Then
there is an additive functor
\begin{align*}
  \widetilde{F}: & \ \ \xymatrix{ \widetilde{\mathscr{C}} \ar[r]& \widetilde{\mathscr{D}}} \\
   & \xymatrix@C=0.6cm{(A,e)\ar@{|->}[r] \ar[d]^\alpha &  (FA,Fe)\ar[d]^{F\alpha}\\
                        (A',e')\ar@{|->}[r]  &  (FA',Fe').}
\end{align*}

Combining the functor $\iota_{\mathscr{C}}: \mathscr{C}\rightarrow\widetilde{ \mathscr{C}}$ given by $A\mapsto (A,\id_{A})$, one has a commutative diagram as follows:
$$
\xymatrix{\mathscr{C}\ar[r]^-{\iota_{\mathscr{C}}}\ar[d]^-F&\widetilde{ \mathscr{C}}\ar[d]^-{\widetilde{F}}\\
\mathscr{D}\ar[r]^-{\iota_{\mathscr{D}}}&\widetilde{\mathscr{D}}.}
$$

\begin{lemma}\label{5.1}
Let $\mathscr{C}$ and $\mathscr{D}$ be extriangulated categories, and $F:\mathscr{C}\to \mathscr{D}$ an exact functor. Then
the induced functor $\widetilde{F}:\widetilde{\mathscr{C}}\to \widetilde{\mathscr{D}}$ is exact.
\end{lemma}

\begin{proof}
Let $\vartriangle$ be an $\mathbb{F}$-triangle in $\widetilde{\mathscr{C}}$. Then there exists an $\mathbb{F}$-triangle $\vartriangle'$ in $\widetilde{\mathscr{C}}$ such that $\vartriangle\oplus \vartriangle'$ is an $\mathbb{E}$-triangle in $\mathscr{C}$. Since $F$ is exact,
$F(\vartriangle\oplus \vartriangle')\cong \widetilde{F}(\vartriangle)\oplus \widetilde{F}(\vartriangle')$ is an $\mathbb{E}$-triangle in $\mathscr{D}$. Thus
$\widetilde{F}(\vartriangle) $ is an $\mathbb{F}$-triangle in $\widetilde{\mathscr{D}}$, which shows that $\widetilde{F}$ is exact.
\end{proof}

Now given two additive functors $F,G$ and a natural transformation $\eta$ as
\begin{equation*}\label{T1}
 \xymatrix@C=0.5cm@R=0.02cm{&\ar@{=>}[dd]^\eta&\\
\mathscr{C}\ar@/^1pc/[rr]^F\ar@/_1pc/[rr]_G && \mathscr{D}.\\
&&}
\end{equation*}
Then there is a natural transformation $\widetilde{\eta}$ as
\begin{equation*}\label{T2}
 \xymatrix@C=0.5cm@R=0.005cm{&\ar@{=>}[dd]^{\widetilde{\eta}}&\\
\widetilde{\mathscr{C}}\ar@/^1pc/[rr]^{\widetilde{F}}\ar@/_1pc/[rr]_{\widetilde{G}} && \widetilde{\mathscr{D}}\\
&&}
\end{equation*}
satisfying $\widetilde{\eta}_{(A,e)}=G(e)\circ \eta_A\circ F(e)$ for any $(A,e)\in\widetilde{\mathscr{C}}$.

\begin{lemma}\label{5.2}
Keep the notations as above. Let $\mathscr{C}$ and $\mathscr{D}$ be extriangulated categories. If $(F,G)$ is an adjoint pair of exact functors, then so is $(\widetilde{F},\widetilde{G})$.
\end{lemma}

\begin{proof}
Assume that $(F,G)$ is an adjoint pair. Then for any $A\in\mathscr{C}$ and any $B\in\mathscr{D}$, there is a natural isomorphism
$$
\eta_{A,B}:\mbox{Hom}_{\mathscr{D}}(FA,B)\overset{\simeq} \longrightarrow \mbox{Hom}_{\mathscr{C}}(A,GB).
$$
We claim that there is a induced natural isomorphism $$
\widetilde{\eta}_{A,B}:\mbox{Hom}_{\widetilde{\mathscr{D}}}(\widetilde{F}(A,e),(B,f)) \longrightarrow \mbox{Hom}_{\widetilde{\mathscr{C}}}((A,e),\widetilde{G}(B,f))
$$
 for any $(A,e)\in\widetilde{\mathscr{C}}$ and any $(B,f)\in\widetilde{\mathscr{D}}$.
Indeed, for any $\alpha\in \mbox{Hom}_{\widetilde{\mathscr{D}}}(\widetilde{F}(A,e),(B,f))$, we have
\begin{align*}
  \alpha\circ F(e)=\alpha=f\circ \alpha & \Rightarrow \ {\eta}_{A,B}(\alpha)\circ e={\eta}_{A,B}(\alpha)=G(f)\circ {\eta}_{A,B}(\alpha) \\
   & \Rightarrow \ {\eta}_{A,B}(\alpha)\in \mbox{Hom}_{\widetilde{\mathscr{C}}}((A,e),\widetilde{G}(B,f)).
\end{align*}
Moreover, let $\alpha\in\mbox{Hom}_{\mathscr{D}}(FA,B)$ and ${\eta}_{A,B}(\alpha)\in \mbox{Hom}_{\widetilde{\mathscr{C}}}((A,e),\widetilde{G}(B,f))$. Then
\begin{align*}
  {\eta}_{A,B}(\alpha)\circ e={\eta}_{A,B}(\alpha)=G(f)\circ {\eta}_{A,B}(\alpha) & \Rightarrow \ \alpha\circ F(e)=\alpha=f\circ \alpha \\
   & \Rightarrow \ \alpha\in \mbox{Hom}_{\widetilde{\mathscr{D}}}(\widetilde{F}(A,e),(B,f)).
\end{align*}
This finishes the proof.
\end{proof}

\begin{definition}\label{5.3}
Let $\mathscr{C}$, $\mathscr{C}'$ and $\mathscr{C}''$ be three extriangulated categories. A \emph{right recollement} of $\mathscr{C}$ relative to
$\mathscr{C}'$ and $\mathscr{C}''$ is a diagram
$$
\xymatrix@C=1.5cm{
\mathscr{C}'\ar@/^0.8pc/[r]|{i_!} & \mathscr{C}\ar@/^0.8pc/[l]|{i^!}\ar@/^0.8pc/[r]|{j^\dag} & \mathscr{C}''\ar@/^0.8pc/[l]|{j_\dag}
}
$$
given by exact functors $i^!$, $i_!$, $j^\dag$ and $j_\dag$ which satisfies the following conditions:
\begin{itemize}
  \item [(RR1)] $(i_!, i^!)$ and $(j^\dag, j_\dag)$ are adjoint pairs.
  \item [(RR2)] $i^! j_\dag=0$.
  \item [(RR3)] $i_!$ and $j_\dag$ are fully faithful.
  \item [(RR4)] For each $A\in\mathscr{C}$, there is an $\mathbb{E}$-triangle
  $$
  \xymatrix{i_!i^!A\ar[r]^-{\theta_A}&A\ar[r]^-{\vartheta_A}&j_\dag j^\dag A\ar@{-->}[r]^-\delta&}
  $$
  in $\mathscr{C}$, where $\theta_A$ and $\vartheta_A$ are given by the adjunction morphisms.
\end{itemize}
\end{definition}

Now we show that the idempotent completion of a right recollement of extriangulated categories is still a right recollement.

\begin{theorem}\label{5.4}
Let $\mathscr{C}$, $\mathscr{C}'$ and $\mathscr{C}''$ be three extriangulated categories. Assume that there is a right recollement
$$
\xymatrix@C=1.5cm{
\mathscr{C}'\ar@/^0.8pc/[r]|{i_!} & \mathscr{C}\ar@/^0.8pc/[l]|{i^!}\ar@/^0.8pc/[r]|{j^\dag} & \mathscr{C}''.\ar@/^0.8pc/[l]|{j_\dag}
}
$$
Then $\widetilde{\mathscr{C}}$ admits a right recollement relative to  $\widetilde{\mathscr{C}'}$ and $\widetilde{\mathscr{C}''}$
$$
\xymatrix@C=1.5cm{
\widetilde{\mathscr{C}'}\ar@/^0.8pc/[r]|{\widetilde{i_!}} & \widetilde{\mathscr{C}}\ar@/^0.8pc/[l]|{\widetilde{i^!}}\ar@/^0.8pc/[r]|{\widetilde{j^\dag}} & \widetilde{\mathscr{C}''}.\ar@/^0.8pc/[l]|{\widetilde{j_\dag}}
}
$$
\end{theorem}

\begin{proof}
$(\RR1)$ By Lemma \ref{5.2}, $(\widetilde{i_!}, \widetilde{i^!})$ and $(\widetilde{j^\dag}, \widetilde{j_\dag})$ are adjoint pairs of exact functors.

$(\RR2)$ By assumption, $i^! j_\dag=0$. Thus $\widetilde{i^!}\widetilde{j_\dag}=\widetilde{i^! j_\dag}=0$.

$(\RR3)$ Since we can view each object of $\widetilde{\mathscr{C}'}$ as a direct summand of ${\mathscr{C}'}$, it follows from \cite[Lemma 11]{CT} that
  $\widetilde{i_!}$ is fully faithful. Similarly, $\widetilde{j_\dag}$ is fully faithful.

$(\RR4)$ Let $(A,e)\in \widetilde{\mathscr{C}}$. Then there is a commutative diagram
  $$
  \xymatrix{i_!i^!A\ar[r]^-{\theta_A}\ar[d]_{i_!i^!e}&A\ar[r]^-{\vartheta_A}\ar[d]^e&j_\dag j^\dag A\ar@{-->}[r]^-\delta\ar[d]^{j_\dag j^\dag e}&\\
  i_!i^!A\ar[r]^-{\theta_A}&A\ar[r]^-{\vartheta_A}&j_\dag j^\dag A\ar@{-->}[r]^-\delta&}
  $$
  in $\mathscr{C}$. In particular, $(i_!i^!e)_\bullet \delta=(j_\dag j^\dag e)^\bullet\delta$. Setting $A_1=(A,e)$ and $A_2=(A, \mbox{id}_A-e)$. Then
  we have the following commutative diagram
  $$
  \xymatrix@C=1.5cm@R=1.3cm{i_!i^!A\ar[r]^-{\theta_A}\ar[d]_{\tiny\begin{pmatrix}i_!i^!e\\i_!i^!(\mbox{id}_A-e)\end{pmatrix}}&
  A\ar[r]^-{\vartheta_A}\ar[d]|{\tiny\begin{pmatrix}e\\\mbox{id}_A-e\end{pmatrix}}
  &
  j_\dag j^\dag A\ar[d]\ar[d]^{\tiny\begin{pmatrix}j_\dag j^\dag e\\j_\dag j^\dag(\mbox{id}_A-e)\end{pmatrix}}&\\
  \widetilde{i_!}\widetilde{i^!}A_1\oplus\widetilde{i_!}\widetilde{i^!}A_2\ar[r]^-{\tiny\begin{pmatrix} \widetilde{\theta}_{A_1}&0\\0&\widetilde{\theta}_{A_2}\end{pmatrix}}&A_1\oplus A_2\ar[r]^-{{\tiny\begin{pmatrix} \widetilde{\vartheta}_{A_1}&0\\0&\widetilde{\vartheta}_{A_2}\end{pmatrix}}}&\widetilde{j_\dag} \widetilde{j^\dag} A_1\oplus \widetilde{j_\dag} \widetilde{j^\dag} A_2,&}
  $$
  where the vertical morphisms are isomorphisms by \cite[Lemma 12]{CT}. Moreover, by Lemma \ref{11},
  $$
  \xymatrix@C=1.3cm{\widetilde{i_!}\widetilde{i^!}A_1\oplus\widetilde{i_!}\widetilde{i^!}A_2\ar[r]^-{\tiny\begin{pmatrix} \widetilde{\theta}_{A_1}&0\\0&\widetilde{\theta}_{A_2}\end{pmatrix}}&A_1\oplus A_2\ar[r]^-{{\tiny\begin{pmatrix} \widetilde{\vartheta}_{A_1}&0\\0&\widetilde{\vartheta}_{A_2}\end{pmatrix}}}&\widetilde{j_\dag} \widetilde{j^\dag} A_1\oplus \widetilde{j_\dag} \widetilde{j^\dag} A_2\ar@{-->}[r]^-{\delta'}&}
  $$
  is an $\mathbb{E}$-triangle,
  where
  \begin{align*}
    \delta' & ={(j_\dag j^\dag e, j_\dag j^\dag(\mbox{id}_A-e))^\bullet}{\begin{pmatrix}i_!i^!e\\i_!i^!(\mbox{id}_A-e)\end{pmatrix}}_\bullet \delta \\
     & =\begin{pmatrix} (i_!i^!e)_\bullet(j_\dag j^\dag e)^\bullet\delta &  (i_!i^!e)_\bullet(j_\dag j^\dag (\mbox{id}_A-e))^\bullet\delta\\
                         (i_!i^!(\mbox{id}_A-e))_\bullet(j_\dag j^\dag e)^\bullet\delta &  (i_!i^!(\mbox{id}_A-e))_\bullet(j_\dag j^\dag (\mbox{id}_A-e))^\bullet\delta \end{pmatrix}.
  \end{align*}
Here,
\begin{align*}
  (i_!i^!e)_\bullet(j_\dag j^\dag (\mbox{id}_A-e))^\bullet\delta & =(j_\dag j^\dag (\mbox{id}_A-e))^\bullet(i_!i^!e)_\bullet\delta \\
   & =(j_\dag j^\dag (\mbox{id}_A-e))^\bullet (j_\dag j^\dag e)^\bullet\delta\\
   &=(j_\dag j^\dag(e(\mbox{id}_A-e)))^\bullet\delta\\
   &=0.
\end{align*}
Similarly, $(i_!i^!(\mbox{id}_A-e))_\bullet(j_\dag j^\dag e)^\bullet\delta=0$.
Thus
$$\delta'=\begin{pmatrix} (i_!i^!e)_\bullet(j_\dag j^\dag e)^\bullet\delta &  0\\
                        0 &  (i_!i^!(\mbox{id}_A-e))_\bullet(j_\dag j^\dag (\mbox{id}_A-e))^\bullet\delta \end{pmatrix}.$$
Set ${\delta}_{A_1}=(i_!i^!e)_\bullet(j_\dag j^\dag e)^\bullet\delta$ and ${\delta}_{A_2}=(i_!i^!(\mbox{id}_A-e))_\bullet(j_\dag j^\dag (\mbox{id}_A-e))^\bullet\delta$.                     Then $$\xymatrix{\widetilde{i_!}\widetilde{i^!}A_1\ar[r]^-{\widetilde{\theta}_{A_1}}&A_1\ar[r]^-{\widetilde{\vartheta}_{A_1}}&\widetilde{j_\dag} \widetilde{j^\dag} A_1\ar@{-->}[r]^-{{\delta}_{A_1}}&}$$
and
$$
\xymatrix{\widetilde{i_!}\widetilde{i^!}A_2\ar[r]^-{\widetilde{\theta}_{A_2}}&A_2\ar[r]^-{\widetilde{\vartheta}_{A_2}}&\widetilde{j_\dag} \widetilde{j^\dag} A_2\ar@{-->}[r]^-{{\delta}_{A_2}}&}$$
are {$\mathbb{F}$-triangles in $\widetilde{\mathscr{C}}$}. Using a similar argument to that of \cite[Remark 10]{CT}, we know that $\widetilde{\theta}_{A_1}$, $\widetilde{\theta}_{A_2}$, $\widetilde{\vartheta}_{A_1}$ and $\widetilde{\vartheta}_{A_2}$ are the adjunction morphisms. This completes the proof.
\end{proof}

Dually, one has the notion of left recollements of extriangulated categories.

\begin{definition}
A \emph{recollement} of extriangulated categories is a diagram
$$
  \xymatrix{\mathscr{C}'\ar[rr]|{i_{\dag}=i_!}&&\ar@/_1pc/[ll]|{i^{\dag}}\ar@/^1pc/[ll]|{i^{!}}\mathscr{C}
\ar[rr]|{j^!=j^{\dag}}&&\ar@/_1pc/[ll]|{j_{!}}\ar@/^1pc/[ll]|{j_{\dag}}\mathscr{C}''}
$$
of extriangulated categories and exact functors such that
$$
\xymatrix@C=1.5cm{
\mathscr{C}'\ar[r]|{i_!} & \mathscr{C}\ar@/^0.8pc/[l]|{i^!}\ar[r]|{j^\dag} & \mathscr{C}''\ar@/^0.8pc/[l]|{j_\dag}
}
$$
is a right recollement, and
$$
\xymatrix@C=1.5cm{
\mathscr{C}'\ar[r]|{i_{\dag}} & \mathscr{C}\ar@/_0.8pc/[l]|{i^\dag}\ar[r]|{j^!} & \mathscr{C}''\ar@/_0.8pc/[l]|{j_!}
}
$$
is a left recollement.
\end{definition}

Following Theorem \ref{5.4}
 and its dual, we have the following

\begin{theorem}
Let $\mathscr{C}$, $\mathscr{C}'$ and $\mathscr{C}''$ be three extriangulated categories. Assume that there is a recollement
$$
  \xymatrix{\mathscr{C}'\ar[rr]|{i_{\dag}=i_!}&&\ar@/_1pc/[ll]|{i^{\dag}}\ar@/^1pc/[ll]|{i^{!}}\mathscr{C}
\ar[rr]|{j^!=j^{\dag}}&&\ar@/_1pc/[ll]|{j_{!}}\ar@/^1pc/[ll]|{j_{\dag}}\mathscr{C}''.}
$$
Then $\widetilde{\mathscr{C}}$ admits a  recollement relative to  $\widetilde{\mathscr{C}'}$ and $\widetilde{\mathscr{C}''}$
$$
  \xymatrix{\widetilde{\mathscr{C}'}\ar[rr]|{\widetilde{i_{\dag}}=\widetilde{i_!}}&&\ar@/_1pc/[ll]|{\widetilde{i^{\dag}}}
  \ar@/^1pc/[ll]|{\widetilde{i^{!}}}\widetilde{\mathscr{C}}
\ar[rr]|{\widetilde{j^!}=\widetilde{j^{\dag}}}&&\ar@/_1pc/[ll]|{\widetilde{j_{!}}}\ar@/^1pc/[ll]|{\widetilde{j_{\dag}}}\widetilde{\mathscr{C}''}.}
$$
\end{theorem}

\section*{Acknowledgments}
The third author is grateful to Panyue Zhou for his helpful suggestions and comments. This work is supported partially by the National Natural Science Foundation of China (No.s 11801273,
11901341).

\end{document}